\documentclass[a4paper,10pt]{article}
\usepackage[utf8]{inputenc}
\usepackage{import}

\usepackage{fullpage}

\usepackage{layouts}

\usepackage[english]{babel}

\usepackage{amsmath,amssymb,amsthm}
\usepackage{mathtools}
\usepackage{mathrsfs}
\usepackage{xfrac}

\usepackage{pdflscape}
\usepackage{tabularx}
\usepackage{booktabs}         

\usepackage{multicol}
\usepackage{multirow}
\usepackage{xparse}
\usepackage{xspace}
\usepackage{makecell}

\usepackage[usenames,dvipsnames,table]{xcolor}

\usepackage{footnote} 
\usepackage{graphicx}
\usepackage{placeins}
\usepackage{svg}
\svgpath{{./images/}}

\definecolor{ultramarine}{RGB}{0,32,96}
\usepackage[colorlinks=true]{hyperref}         
\usepackage{cleveref}
\usepackage{caption}          

\usepackage[notref,final]{showkeys}
\usepackage[color=gray, backgroundcolor=gray!10, textwidth=2cm, textsize=footnotesize]{todonotes}
\usepackage{setspace}

\usepackage{subcaption}

\usepackage{framed} 

\usepackage{enumitem}

\definecolor{DarkerGreen}{RGB}{0,170,0}
\definecolor{DarkerRed}{RGB}{170,0,0}

\definecolor{myRed}{rgb}{0.450385, 0.157961, 0.217975}
\definecolor{myBlue}{rgb}{0.139681, 0.311666, 0.550652}

\definecolor{myKcolor}{rgb}{0, 0.411765, 0.572549}
\definecolor{myEcolor}{rgb}{0, 0, 0}

\newcommand{\Fig}[1]{Fig.~\ref{#1}}

\NewDocumentCommand\sstr{o}{\IfNoValueTF{#1}{\sigma}{\sigma\ar{#1}}}
\newcommand{\sbra}[1]{\left[ #1 \right]}

\newcommand\res{\mathop{\hbox{\vrule height 7pt width .5pt depth 0pt
			\vrule height .5pt width 6pt depth 0pt}}\nolimits}

\newcommand\eps{\varepsilon}

\newcommand\R{\mathbb{R}}
\newcommand\N{\mathbb{N}}

\newcommand\calG{\mathcal{G}}

\newcommand\calA{\mathcal{A}}
\newcommand\calL{\mathcal{L}}
\newcommand\calH{\mathcal{H}}

\newcommand\callF{\mathcal{F}}

\newcommand\calulu{\mathfrak{U}}
\newtheorem{theorem}{Theorem}[section]

\newtheorem{proposition}[theorem]{Proposition}
\newtheorem{lemma}[theorem]{Lemma}
\newtheorem{remark}[theorem]{Remark}
\newtheorem{corollary}[theorem]{Corollary}

\numberwithin{equation}{section}
\usepackage{fancyhdr}
\newcommand\Functeps{\mathcal F_\eps}
\newcommand\Functepstilde{\widetilde{\mathcal F}_\eps}

\newcommand\dx{{\mathrm d}x}

\newcommand\dt{{\mathrm d}t}
\newcommand\dd{{\mathrm d}}

\newcommand\PotDann{\omega}
\newcommand\FailureS{\overline{\varsigma}}
\newcommand{\hsigmabarra}{h_{\FailureS}}
\newcommand\Functepsk{\mathcal F_{\eps_k}}
\newcommand\Functepskj{\mathcal F_{\eps_{k_j}}}

\newcommand\Functu{\mathscr{F}}
\newcommand\FunctepsDir{\mathcal D_\eps}
\newcommand\hatf{l}

\graphicspath{
  {images/}
  {images/mathematica/}
  {/home/roberto/Insync/OneDrive/UniPi/ZoomNotes-backup/Research/Fracture/}}

\begin{document}

\begin{center}
  {\Large
{Phase-field modelling of cohesive fracture. Part I: $\Gamma$-convergence results}}\\[5mm]
{\today}\\[5mm]
Roberto Alessi$^{1}$, Francesco Colasanto$^{2}$ and Matteo Focardi$^{2}$\\[2mm]
{\em $^{1}$ DICI, Università di Pisa,\\ 56122 Pisa, Italia}\\[1mm]
{\em $^{2}$ DiMaI, Universit\`a di Firenze\\ 50134 Firenze, Italy}
\\[3mm]
\begin{minipage}[c]{0.8\textwidth}
The main aim of this three-part work is to provide a unified consistent framework for the phase-field modeling of cohesive fracture.
In this first paper we establish the mathematical foundation of a cohesive phase-field model by proving a $\Gamma$-convergence result in a one-dimensional setting. Specifically, we consider a broad class of phase-field energies, encompassing different models present in the literature, thereby both extending the results in \cite{ContiFocardiIurlano2016} and providing an analytical validation of all the other approaches. Additionally, by modifying the functional scaling, we demonstrate that our formulation also generalizes the Ambrosio-Tortorelli approximation for brittle fracture, therefore laying the groundwork for a unified framework for variational fracture problems.
The Part~II paper presents a systematic procedure for constructing phase-field models that reproduce prescribed cohesive laws,
whereas the Part~III paper validates the theoretical results with applied examples.
\end{minipage}
\end{center}

\tableofcontents

\section{Introduction}
\label{sec_cpflit}
\FloatBarrier

\subsection{Background and Motivation}

Understanding and predicting fracture initiation and propagation is essential for designing safer and more resilient mechanical systems while preventing structural failures. Over the past century, Griffith’s brittle fracture theory \cite{Griffith1921} has become one of the most fundamental and widely used models in fracture mechanics. In Griffith’s framework, fracture energy is assumed to be dissipated entirely upon the formation of a crack, with no residual forces acting between the crack surfaces, regardless of the displacement jump across them (\Fig{fig_BCL}).

Despite its simplicity and historical significance, Griffith’s model presents two key limitations: (i)~it cannot predict crack initiation in an initially pristine elastic body, requiring a pre-existing defect \cite{marigo2010,Tanne2018,Kumar2020}, and (ii)~it leads to unrealistic scale effects in fracture predictions \cite{Marigo2023}. Cohesive fracture models overcome these issues by admitting nonzero forces, the cohesive forces, between crack surfaces. These models, originally proposed by Dugdale and Barenblatt \cite{Dugdale1960,Barenblatt1962}, define a surface energy density that depends on the displacement jump, providing a more realistic representation of fracture processes (\Fig{fig_BCL}). Characterized by a critical stress and a characteristic length, cohesive laws effectively address the deficiencies of Griffith’s model \cite{Marigo2023}.

Building on the variational reformulation of Griffith’s fracture as a free-discontinuity energy minimization problem \cite{Marigo1998,Braides95}, cohesive fracture models have also been recast within a similar variational framework \cite{Bourdin2008}.
\begin{figure}[h!]
 \centering
 \small
  \includegraphics[width=\linewidth]{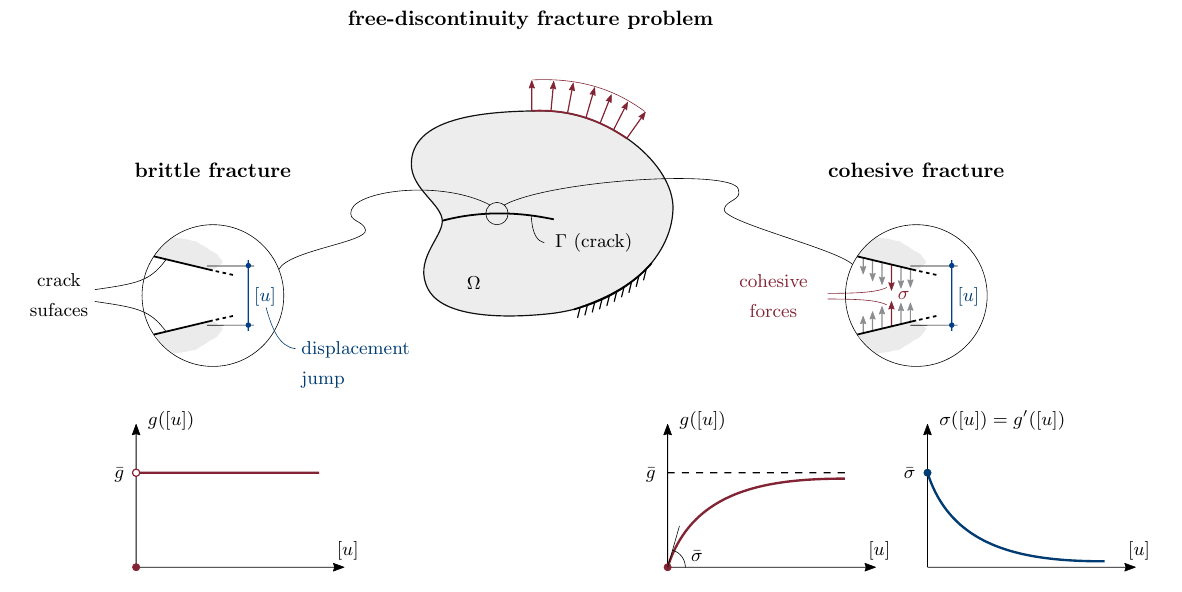}
 \caption{Paradigmatic fracture problem: qualitative trends of the surface energy density $g$ and associated cohesive stress $\sstr$ with respect to the crack opening (displacement jump) $\sbra{u}$ for brittle and cohesive fracture models.}
 \label{fig_BCL}
\end{figure}
The free-discontinuity brittle fracture energy minimization problem admits a regularization that has been explored by the first time in \cite{Bourdin2000b,Focardi2001}, drawing inspiration from the work of \cite{Ambrosio1992} (see also \cite{AT90}).
Within the context of $\Gamma$-convergence, the crack path emerges from the localization of a smooth phase-field as the regularization length approaches zero \cite{Braides1998}.
The regularization of brittle fracture via phase-field models has enabled the numerical simulation of complex fracture processes that were impractical with classical methods \cite{Bourdin2014,Mesgarnejad2015}. Today, the phase-field approach is a leading strategy in fracture mechanics, allowing for the capture of crack nucleation, both with and without pre-existing notches, and the modeling of intricate crack patterns. Its simple numerical implementation, often done by alternate minimization schemes, has also contributed to its widespread diffusion.
The work of \cite{Bourdin2000b} provides a key link between the mathematical results and engineering application within the fracture mechanics field.

However, phase-field models for brittle fracture inherit a critical limitation from Griffith's model: the inability to independently control the critical strength, the fracture toughness (or critical energy release rate), and the regularization length.
As a result, the nucleation stress threshold cannot be directly linked to the sharp interface of Griffith's model. This limitation hinders the development of a flexible and general model capable of accurately describing crack nucleation in smooth or notched domains \cite{Bourdin2008,Tanne2018,Lopez-Pamies2024}.

As noted in \cite{Marigo2023}, cohesive fracture models are a natural way to address the shortcomings of Griffith’s brittle fracture theory. Inspired by \cite{Pham2010c,Pham2010a}, \cite{Lorentz2011} developed a standalone gradient-damage (phase-field) model, which is not derived from a free-discontinuity fracture problem. This model was first used to study cohesive fracture in a one-dimensional setup, providing closed-form solutions. A higher-dimensional extension of this model applied to large-scale simulations of failure processes has been accomplished in \cite{Lorentz2011a}.
Despite its potential, this cohesive phase-field model initially received little attention from the fracture mechanics community.

A successful attempt to create a mathematically consistent variational phase-field model regularizing the free-discontinuity cohesive fracture problem in \cite{Bourdin2008} was made by \cite{ContiFocardiIurlano2016} (see also \cite{ContiFocardiIurlano2022,ContiFocardiIurlanopgrowth} for the vector valued analogues).
This work, like \cite{Bourdin2000b}, represents a valuable intersection between the mathematical and engineering communities.
The subsequent numerical implementation by \cite{Freddi2017} faced several challenges, such as developing a backtracking algorithm and further regularizing the degradation function with respect to the internal length. Additionally, the difficulty in tuning the elastic degradation function to match a specific cohesive law, combined with the use of a fixed quadratic phase-field dissipation function, likely limited the model's flexibility and hindered its broader use within the fracture mechanics community.

Significant advancements were made by \cite{Wu2017,Wu2018b}, who, building on \cite{Lorentz2011}, introduced a new generation of phase-field cohesive fracture models. Also these models are grounded in the well-established variational framework of gradient-damage models \cite{Pham2010c,Pham2010a}. By incorporating polynomial crack geometric functions and rational energetic degradation functions, they allow for the independent tuning of critical stress, fracture toughness, and regularization length to match specific softening laws. A key aspect in these models, adopted by \cite{ContiFocardiIurlano2016,Wu2017,Wu2018b}, is the inclusion of the regularization length within the elastic degradation function, enabling a more precise description of cohesive fracture failures.

Despite their flexibility and soundness, the process of setting the material functions to achieve a target softening law remained unclear. A major contribution in this regard came from \cite{Feng2021}, who introduced an integral relation that links a single unknown function, defining both the degradation and phase-field dissipation functions, to the desired traction-separation law.

It is worth emphasising that despite claims of $\Gamma$-convergence in \cite[Sec. 5.1]{Wu2017} and in \cite[Sec. 6]{Feng2021}, these efforts primarily demonstrated numerical convergence rather than a rigorous mathematical proof.
It is then clear that an analytical proof of the $\Gamma$-convergence of \cite{Wu2017,Wu2018b} and \cite{Feng2021}
phase-field models towards the free-discontinuity variational cohesive fracture model set
up by \cite{Bourdin2008} is still missing, in the same spirit of what has been done by
\cite{ContiFocardiIurlano2016} (cf. also \cite{ContiFocardiIurlano2022,ContiFocardiIurlanopgrowth,Colasanto2024}).

A timeline, visually summarizing the state of the art discussed above, is presented in the introduction of the third part paper, \cite{Alessi2025c}, where the key milestone works with their logical connections are highlighted, including the present study, within the variational approach to fracture.

\bigskip

The main aim of this three-part work is to provide a unified consistent framework for the phase-field modeling of cohesive fracture, similar to what has been done in \cite{Bourdin2008} for the the phase-field modeling of brittle fracture, that bridges the mathematical results of the first two parts with the more applied and engineering oriented third part.
One of the main goals of this first part work is exactly to fill the gap discussed above by extending the $\Gamma$-convergence
results in \cite{ContiFocardiIurlano2016}, thus proposing a unified mathematically foundation for cohesive phase
field fracture models.
More precisely, in this first paper we prove a $\Gamma$-convergence result in the one-dimensional setting
(the general case is addressed in \cite{Colasanto2024}, see the discussion after
Theorem~\ref{t:finale brittle}) for a very general family
of phase-field energies which encompass at the same time the models introduced in \cite{ContiFocardiIurlano2016}, \cite{Wu2017,Wu2018b}, \cite{Feng2021} and in \cite{Lammen2023,Lammen2025} (see Section~\ref{s:model} below). In addition, changing the scaling in the functionals, we show that our model provides also a generalization of the classical Ambrosio-Tortorelli model for the approximation of brittle energies.
In the second paper \cite{Alessi2025b}, we take advantage of the general model introduced in this first part work to
set-up an analytical procedure to construct the phase-field model in order to obtain assigned cohesive laws, either by fixing the degradation function and choosing the damage potential or viceversa. This validates rigorously and generalizes the numerical results of \cite{Feng2021} allowing, for instance, to derive different phase-field models sharing the same target cohesive law, and hence exhibiting the same overall cohesive fracture behaviour, but different localized phase-field profiles evolutions.
In the third paper \cite{Alessi2025c}, the mechanical responses of different phase-field models associated to canonical traction-separation laws
in a one-dimensional setting are investigated in depth under a more engineering oriented perspective providing a link with and supporting the conclusions of the theoretical results of the first two parts.

\subsection{A general phase-field model}\label{s:model}

To introduce the mentioned new phase-field functionals we recall both
the model introduced in \cite{ContiFocardiIurlano2016} (actually a slight generalization of it),
and its regularization proposed in \cite{Wu2017,Wu2018b} and later specified in \cite{Feng2021}.
To this aim we fix some notation and state some assumptions that will be used throughout the paper.
For every $t\in[0,1)$ let
\begin{equation}\label{e:f}
f(t):=\left( \frac{\hatf(t)}{Q(1-t)}\right)^{\sfrac12}\,,
\end{equation}
where
\begin{itemize}
\item[(Hp~$1$)] $\hatf,\, Q\in C^0([0,1],[0,\infty))$ with $\hatf (1)= 1$,
$\hatf^{-1}(\{0\})=Q^{-1}(\{0\})=\{0\}$, $Q$ non-decreasing in a right neighbourhood of the origin,
and such that $[0,1)\ni t\mapsto f(t)$ is non-decreasing in a left neighbourhood of $t=1$.
\end{itemize}
Clearly, assuming  $\hatf (1)= 1$ is not restrictive up to a change of $Q$ by a multiplicative factor.
Then, for every $\eps>0$, set $\widetilde{f}_\eps(1):=1$, for $t\in[0,1)$
\begin{equation}\label{e:f eps tilde}
    \widetilde{f}_\eps(t):=1\wedge \eps^{\sfrac12}f(t)\,,
\end{equation}
and
\begin{equation}\label{e:f eps}
f_\eps(t):=\left( \frac{\eps \hatf(t)}{\eps \hatf(t)
+Q(1-t)}\right)^{\sfrac12}\,.
\end{equation}
Note that
$[0,1]\ni t\mapsto f_{\eps}(t)$ is non-decreasing in a left neighbourhood of $t=1$, as well.
Consider next
\begin{itemize}
\item[(Hp~$2$)] $\PotDann\in C^0([0,1],[0,\infty))$ such that $\PotDann^{-1}(\{0\})=\{0\}$,
and the following limit exists
\begin{equation}\label{e:FailureS def}
\lim_{t\to 0^+}\left(\frac{\PotDann(t)}{Q(t)}\right)^{\sfrac12}
=:\FailureS\in[0,\infty]\,.
 \end{equation}
\end{itemize}
Let $\Omega\subset\R$ be a bounded and open interval, and
$\calA(\Omega)$ the family of its open subsets.
Then define $\Functeps^{(1)},\Functeps^{(2)}:L^1(\Omega,\R^2)\times\calA(\Omega)\to[0,\infty]$ respectively by
\begin{equation}\label{functeps CFI}
 \Functeps^{(1)}(u,v,A):= \int_A \left(
 \widetilde{f}_\eps^2(v) |u'|^2 + \frac{\PotDann(1-v)}{4\eps} + \eps |v'|^2 \right) \dx\,
\end{equation}
and
\begin{equation}\label{functeps Wu}
 \Functeps^{(2)}(u,v,A):= \int_A \left( f_\eps^2(v) |u'|^2 + \frac{\PotDann(1-v)}{4\eps} + \eps |v'|^2 \right) \dx
\end{equation}
if $(u,v)\in H^1(\Omega, \R\times[0,1])$ and $\infty$ otherwise.

Some remarks are in order: the original model in \cite{ContiFocardiIurlano2016} corresponds to
$\widetilde{f}_\eps$ defined with $\hatf(t)$ as above, $Q(t)=\lambda^{-2}t^{2q}$,
$q\geq 1$, $\lambda\in(0,\infty)$, and $\PotDann(t)=t^2$ in \eqref{e:f eps tilde} (see also \cite{ContiFocardiIurlano2022,ContiFocardiIurlanopgrowth}).
Therefore, $\FailureS=\lambda$ if $q=1$, and $\FailureS=\infty$ otherwise.
Instead, the model proposed in \cite{Wu2017,Wu2018b} corresponds to $\hatf(t):=t^p$ with $p\in (0,\infty)$ in \eqref{e:f eps}.

Under the above quoted choices of $Q$ and $\PotDann$, the $\Gamma(L^1)$-convergence of $\{\Functeps^{(1)}\}_{\eps>0}$ to a cohesive type functional has been addressed in \cite[Theorem~2.1]{ContiFocardiIurlano2016} (cf. below for the explicit expression of the limiting functional, and see \cite{ContiFocardiIurlano2022,ContiFocardiIurlanopgrowth} for the vector-valued geometrically nonlinear setting). We prove here an analogous result for the families
$\{\Functeps^{(1)}\}_{\eps>0}$ and
$\{\Functeps^{(2)}\}_{\eps>0}$, and actually for a broader family of functionals including both those defined above and those used in \cite{Lammen2023,Lammen2025} as particular cases. Indeed, let $\varphi:[0,\infty)\to[0,\infty)$ be a non-decreasing
function, and consider the functionals $\Functeps:L^1(\Omega,\R^2)\times\calA(\Omega)\to[0,\infty]$
defined by
\begin{equation}\label{functeps}
 \Functeps(u,v,A):= \int_A \left( \varphi(\eps f^2(v)) |u'|^2 + \frac{\PotDann(1-v)}{4\eps} + \eps |v'|^2 \right) \dx\,
\end{equation}
if $(u,v)\in H^1(\Omega, \R\times[0,1])$ and $\infty$ otherwise. In particular, in formula \eqref{functeps} above, the functions $[0,1)\mapsto\varphi(\eps f^2(t))$ are extended by continuity to $t=1$ with value $\varphi(\infty)$ for every $\eps>0$.
In the main result below we will assume that $\varphi$ satisfies
\begin{itemize}
\item[(Hp~$3$)] $\varphi\in C^0([0,\infty))$, $\varphi$ is non-decreasing and 
$\varphi(\infty):= \displaystyle{\lim_{t\to\infty}\varphi(t)}\in(0,\infty)$;
\item[(Hp~$4$)] $\varphi^{-1}(0)=\{0\}$, and
$\varphi$ is (right-)differentiable in $0$ with $\varphi'(0^+)\in(0,\infty)$.
\end{itemize}
Now let $\Functu_{\FailureS}:L^1(\Omega)\times \mathcal{A}(\Omega)\to [0,\infty]$ be given by
\begin{equation}\label{e:H}
 \Functu_{\FailureS}(u,A):=
 \begin{cases}
\displaystyle{\int_{A}\hsigmabarra^{**}(|u'|)\dx +(\varphi'(0^+))^{\sfrac12}\FailureS|D^cu|(A) +\int_{J_u\cap A}g(|[u]|)\dd \calH ^0} \; \; & \\
   & \hskip-2cm\textup{if} \; u\in GBV(\Omega) \\
\infty & \hskip-2cm\textup{otherwise} \\
 \end{cases}
\end{equation}
where for $\varsigma\in[0,\infty]$ we define $h_\varsigma:[0,\infty) \to [0,\infty)$ by
\begin{equation}\label{e:hsigma}
  h_\varsigma(t):=\inf_{\tau\in(0,\infty)}
\left\{\varphi\left(\frac1\tau\right)t^2 +\frac{\varsigma^2}{4}\tau\right\}\,,
\end{equation}
(in particular, $h_0(t)=0$ and $h_\infty(t)=\varphi(\infty)t^2$ for every $t\in[0,\infty)$, cf.
Proposition~\ref{p:proprieta hsigma} below),
$\hsigmabarra^{**}$ denotes the convex envelope of $\hsigmabarra$, and
$g:[0,\infty)\to \mathbb{R}$ is defined by
\begin{equation}\label{e:lagiyo}
  g(s):=\inf_{(\gamma,\beta)\in \calulu _s}
  \underbrace{\int_0^1 \Big(\PotDann(1-\beta)
  \big((\varphi'(0^+)f^2(\beta)|\gamma'|^2
  +|\beta'|^2\big)\Big)^{\sfrac12}\dx}_{\calG(\gamma,\beta):=}\,,
\end{equation}
where $\calulu_s$ is defined in \eqref{e:glispaziu}.
Throughout the paper we adopt the convention $0\cdot\infty=0$. Therefore,
if $\Functu_\infty(u)<\infty$ then $|D^cu|(\Omega)=0$ (cf. \eqref{e:H}),
and thus we conclude that $u\in GSBV(\Omega)$ with $u'\in L^2(\Omega)$.

We will prove the following result.
\begin{theorem}\label{t:finale}
Assume (Hp~$1$)-(Hp~$4$), and \eqref{e:FailureS def} hold with $\FailureS\in(0,\infty]$.
Let $\Functeps$ be the functional defined in \eqref{functeps}.
Then, for all $(u,v)\in L^1(\Omega,\R^2)$
\[
\Gamma({L^1})\text{-}\lim_{\eps\to0}\Functeps(u,v)=
F_{\,\FailureS}(u,v)\,,
\]
where, $F_{\,\FailureS}:L^1(\Omega,\R^2)\to[0,\infty]$ is defined by
\begin{equation}\label{F0}
F_{\,\FailureS}(u,v):=
\begin{cases}
\Functu_{\,\FailureS}(u) & \textup{if $v=1$ $\mathcal{L}^1$-a.e. on $\Omega$}\cr
\infty & \textup{otherwise}.
\end{cases}\,
\end{equation}
\end{theorem}
The proof of Theorem~\ref{t:finale} follows some of the ideas introduced in \cite{ContiFocardiIurlano2022,ContiFocardiIurlanopgrowth} for the $\Gamma$-convergence analysis of the
geometrically nonlinear counterpart of the model studied in \cite{ContiFocardiIurlano2016} in the scalar case.
Despite this, several nontrivial difficulties have to be overcome due to the generality of the model.

In passing, note that $\Functeps^{(1)}$, $\Functeps^{(2)}$ correspond, respectively, to the choices $\varphi_1(t)=1\wedge t$ and $\varphi_2(t)=\frac t{1+t}$ for every $t\in[0,\infty)$, and $\hatf(t)=t^p$ in the second setting. Therefore, $\varphi_i$, $i\in\{1,2\}$, satisfies (Hp~$3$) and (Hp~$4$), with $\varphi_i(\infty)= \varphi_i'(0)=1$.
Moreover, if $h_\varsigma^{(i)}$ denotes the corresponding function defined in \eqref{e:hsigma}, we have for every $\varsigma\in(0,\infty]$
\begin{equation*}
(h_\varsigma^{(1)})^{**}(t)=h_\varsigma^{(2)}(t)=
(h_\varsigma^{(2)})^{**}(t)=
\begin{cases}
  t^2 \qquad    & \textup{if} \; t\leq \frac{\varsigma}{2} \\
  \varsigma t-\frac{\varsigma^2}{4} & \textup{if} \; t\geq \frac{\varsigma}{2}\,.
  \end{cases}
\end{equation*}
Note that the same bulk energy density is obtained in both cases, without the need of taking the
convexification with the choice $\varphi_2$. We then expect the bulk energy density of the
corresponding $\Gamma$-limit to be convex also in the vector valued setting, contrary to the case
when $\varphi_1$ is chosen (cf. \cite[Section~2.2]{ContiFocardiIurlano2022} and \cite{Colasanto2024}).

In addition, we remark that the dependence of $g$ on $\varphi$ is elementary and explicit.
This claim can be highlighted letting, for instance, $g_1$ be the function corresponding to the choice
$\varphi_1=1\wedge t$ as in \cite{ContiFocardiIurlano2016}, so that \eqref{e:lagiyo}
itself yields for every $s\geq 0$
\begin{equation}\label{e:dependence g varphi}
g(s)=g_1\big((\varphi'(0^+))^{\sfrac12}s\big)\,.
\end{equation}
Note that any other function $\varphi$ with slope equal to $1$ in $t=0$ would work the same.

We point out that the surface energy densities $g$ we obtain can have either a linear or a superlinear behaviour
for small jump amplitudes, the difference between the two cases being encoded by the finiteness or otherwise of
the value of the limit $\FailureS$
(cf. \eqref{e:FailureS def}). The approximation of superlinear surface energy densities slightly departs from the analysis in \cite[Section~7.2]{ContiFocardiIurlano2016}, there the numerator $\hatf$ of the degradation function is $\eps$ dependent in contrast to our approach, and is closer to the approach in \cite{ContiFocardiIurlanopgrowth}.

We underline that another advantage of the introduction in the model of the
function $\varphi$ is that a simple change of the scaling in the functionals $\Functeps$
provide a family of energies including those originally defined by Ambrosio and Tortorelli
in \cite{Ambrosio1992} to approximate the Mumford-Shah energy. Therefore, we provide a unifying
phase-field model for the approximation of Griffith energy in brittle fracture (\cite{Bourdin2008})
and cohesive zone models.
Indeed, let $\gamma_\eps>0$ with $\sfrac{\gamma_\eps}\eps\to\infty$ as $\eps\to 0^+$, and consider the functionals $\Functepstilde:L^1(\Omega,\R^2)\times\calA(\Omega)\to[0,\infty]$ defined by
\begin{equation}\label{functeps brittle}
 \Functepstilde(u,v,A):= \int_A \left( \varphi(\gamma_\eps f^2(v)) |u'|^2 + \frac{\PotDann(1-v)}{4\eps} + \eps |v'|^2 \right) \dx\,
\end{equation}
if $(u,v)\in H^1(\Omega, \R\times[0,1])$ and $\infty$ otherwise.
As a consequence of Theorem~\ref{t:finale} we obtain the following result to state which
it is convenient to introduce the function
$\Psi:[0,1]\to[0,\infty)$ given by
\begin{equation}\label{e:Psi}
\Psi(t):=\int_0^t\PotDann^{\sfrac12}(1-\tau)d\tau\,.
\end{equation}
\begin{theorem}\label{t:finale brittle}
Assume (Hp~$1$)-(Hp~$4$), \eqref{e:FailureS def} hold with $\FailureS\in(0,\infty]$, and that
$\sfrac{\gamma_\eps}\eps\to\infty$ as $\eps\to0^+$.
Let $\Functepstilde$ be the functional defined in \eqref{functeps brittle}.
Then, for all $(u,v)\in L^1(\Omega,\R^2)$
\[
\Gamma({L^1})\text{-}\lim_{\eps\to0}\Functepstilde(u,v)=
\widetilde{F}(u,v)\,,
\]
where, $\widetilde{F}:L^1(\Omega,\R^2)\times\calA(\Omega)\to[0,\infty]$ is defined by
\begin{equation}\label{F0 brittle}
\widetilde{F}(u,v,A):=
\begin{cases}
\displaystyle{\varphi(\infty)\int_{A}|u'|^2\dx +2\Psi(1)\calH^0(J_u\cap A)} \;\; & \\
&\hskip-2cm \textup{if $u\in SBV(\Omega)$, $v=1$ $\mathcal{L}^1$-a.e. on $\Omega$}\\
\infty & \hskip-2cm\textup{otherwise}
\end{cases}\,,
\end{equation}
\end{theorem}

Finally, the role of the assumptions $\{\varphi(\infty),\varphi'(0^+)\}\subset(0,\infty)$ are analysed
in Section~\ref{ss:other regimes} by discussing the complementary cases.

Let us conclude the comments to Theorem~\ref{t:finale} by remarking that the one-dimensional setting is not mandatory for the
$\Gamma$-convergence analysis. Indeed, results similar to those contained in \cite{ContiFocardiIurlano2022} can be obtained in the vector-valued setting (see \cite{Colasanto2024}).
In this respect, \cite[Corollary~3.5]{ContiFocardiIurlano2022} establishes that the surface energy density of the $\Gamma$-limit depends only on the modulus of the jump via the one-dimensional surface energy density obtained in \cite{ContiFocardiIurlano2016}
in the so called isotropic setting, roughly speaking if the (linear recession function
of the) bulk energy density in the phase-field model is the modulus squared of the deformation gradient.
We expect the same to be true in this generalized setting, so that the problem of assigning a specific cohesive law studied in the second part paper \cite{Alessi2025b} can be also addressed for some vector-valued models, as well.

\subsection{Contents of this Paper}

In Section~\ref{s:model} we have introduced the phase-field model under study and the main
$\Gamma$-convergence result has been stated (cf. Theorem~\ref{t:finale}).
The proof of Theorem~\ref{t:finale} is given in Section~\ref{s:gammaconv compactness} where the compactness properties of the phase-field functionals and the $\Gamma$-convergence with the addition of Dirichlet boundary conditions are also studied (cf. Theorems~\ref{t:cptness sigma finito} and \ref{t:cptness sigma infinito}, Theorem~\ref{t:Dirichlet}, respectively). In particular, convergence of minimizers in that setting is established, so that the choice of the $L^1$ topology is fully justified by the latter result.
All the needed technical results to prove the above mentioned statements are collected in Section~\ref{s:technical results}. The approximation of brittle type energies is discussed in Section~\ref{ss:other regimes}, together
with other consequences of Theorem~\ref{t:finale}.

\section{Technical results}\label{s:technical results}

\subsection{Preliminaries}
We adopt standard notation for Sobolev, $BV$, $GBV$ spaces for which we refer to
\cite{AFP}. Let us only recall that if $u\in L^1\cap G(S)BV(\Omega)$, setting
\begin{equation}\label{e:truncations}
    u^M:=(u\wedge M)\vee(-M)\,,
\end{equation}
for $M\in\N$ and $u\in L^1(\Omega)$,
then $u^M\in(S)BV(\Omega)$, $u^M\to u$ in $L^1(\Omega)$ as $M\to\infty$,
$J_u=\cup_{M\in\N}J_{u^M}$,
$[u](x)=[u^M](x)$ for $M$ sufficiently large, $u'=(u^M)'$ $\calL^1$-a.e. on $\{|u|\le M\}$, and
$|D^cu|(A):=\sup_M|D^cu^M|(A)$ for all $A\in\calA(\Omega)$
\cite[Section~4.5]{AFP}. Here, $u'$ denotes the approximate gradient of $u$, $J_u$
is the set of approximate jump points, and $[u]$ is the approximate jump (cf. \cite[Definitions~4.30, 4.31, Theorem~4.34]{AFP})

For the standard theory and results on $\Gamma$-convergence
we refer to \cite{Dalmaso1993, Braides1998}.

In what follows when taking the right-(left-)limit of a monotone function $\psi:(a,b)\to\R$,
where $a\in[-\infty,\infty)$ and $b\in(-\infty,\infty]$, in a point $t_0\in[a,b]$
we will simply write $\psi(t_0^+)$ ($\psi(t_0^-)$) if $t_0$ is finite, and $\psi(\pm\infty)$ otherwise.

\subsection{Technical results for the diffuse part}

We establish in the next statement several useful properties of the functions $h_\varsigma$ for functions $\varphi$ satisfying slightly more general assumptions than (Hp~$3$) and (Hp~$4$).
In particular, if $\varphi(\infty)=\infty$ we set by definition $h_\infty(t)=\infty$ if $t>0$ and
equal to $0$ if $t=0$.
In what follows with $\varphi'(0^+)=\infty$ we mean that the limit of the difference quotient of $\varphi$ in $t=0$ exists and it is not finite.
\begin{proposition}\label{p:proprieta hsigma}
Let $\varphi\in C^0([0,\infty))$, be positive, non-decreasing with $\varphi^{-1}(\{0\})=0$.
Let $h_\varsigma$ be the function defined in \eqref{e:hsigma}, then
\begin{itemize}
\item[(i)] Let $\varsigma\in(0,\infty)$. Then,
$h_\varsigma(0)=0$, $h_\varsigma\in C^0([0,\infty))$, $h_\varsigma$ is strictly increasing, 
$h_0(t)<h_{\varsigma'}(t)\leq h_\varsigma(t)\leq h_\infty(t)$ if $0\leq\varsigma'<\varsigma<\infty$ for every $t>0$. Moreover, for every $t\geq0$ 
we have
$\displaystyle{\lim_{\varsigma'\to\varsigma}h_{\varsigma'}(t)=h_\varsigma(t)}$, and
\begin{equation}\label{e:hsigmacvxapprox}
  \lim_{\varsigma'\to\varsigma}h^{**}_{\varsigma'}(t)=h^{**}_\varsigma(t)\,.
\end{equation}
\item[(ii)] Let $\varsigma\in(0,\infty)$. If $\varphi$ has right derivative $\varphi'(0^+)\in[0,\infty]$ then
\begin{equation}\label{e:laaccabroo}
\lim_{t\to\infty}\frac{h_\varsigma^{**}(t)}t=
\lim_{t\to\infty}\frac{h_\varsigma(t)}t=
(\varphi'(0^+))^{\sfrac12}\varsigma\,.
 \end{equation}
If in addition $\varphi'(0^+)\in[0,\infty)$, then there is $\xi_\varsigma>0$ such that for every $t\in [0,\infty)$
\begin{equation}\label{minorconlin}
((\varphi'(0^+))^{\sfrac12}\varsigma t -\xi_\varsigma)\vee 0\leq h_\varsigma^{**}(t)\leq (\varphi'(0^+))^{\sfrac12}\varsigma t \,.
\end{equation}
Moreover, if $\varphi'(0^+)\in(0,\infty]$
\begin{equation}\label{e:laaccabrodo}
\lim_{t\to0}\frac{h_\varsigma^{**}(t)}{t^2}=\lim_{t\to0}\frac{h_\varsigma(t)}{t^2}=\varphi(\infty)\in
(0,\infty]\,.
 \end{equation}
\item[(iii)] Let $\varsigma=\infty$. If $\varphi(\infty)\in(0,\infty)$ then
$h_\infty(t)=\varphi(\infty)t^2$ for every $t\in [0,\infty)$. If, moreover, $\varphi'(0^+)\in(0,\infty]$ then
for every $t\in [0,\infty)$
\begin{equation}\label{e:hinftyapprox}  \lim_{\varsigma\to\infty} h^{**}_\varsigma(t)=\lim_{\varsigma\to\infty} h_\varsigma(t)
=h_\infty(t)\,.
\end{equation}
\item[(iv)] Let $\varsigma=0$, $h_0(t)=0$ for every $t\in [0,\infty)$
and
\begin{equation}\label{e:h0approx}
  \lim_{\varsigma\to0^+}h^{**}_\varsigma(t)
  =\lim_{\varsigma\to0^+} h_\varsigma(t)
  =h_0(t)\,.
\end{equation}
\end{itemize}
\end{proposition}
\begin{proof} \noindent{\bf Step~1:} Proof of item (i).

It is clear that $h_\varsigma(0)=0$ for every $\varsigma\in[0,\infty]$, and moreover that $h_0(t)=0$ for every $t\geq 0$ as $\varphi(0)=0$. Moreover, by the very definition, if
$\FailureS=\infty$ then $h_\infty(t)=\varphi(\infty)t^2$ for every $t\geq 0$.
Note that the infimum in the definition of $h_\varsigma(t)$, $t>0$, is actually a minimum by the continuity and monotonicity of $\varphi$, and thus $h_\varsigma(t)>0=h_0(t)$.
Denote by $\tau_{\varsigma,t}>0$ a minimizer, then $\tau_{\varsigma,t}\in[0,\frac{4}{\varsigma^2}\varphi(\sfrac1t)t^2+t]$ by choosing $\tau=t$ in the very definition of $h_\varsigma(t)$.
Using $\tau_{\varsigma,t}$ as a test in the definition of $h_\varsigma(s)$ for $s<t$
yields $h_\varsigma(s)<h_\varsigma(t)$.

To establish the equality $\displaystyle{\lim_{\varsigma'\to\varsigma}h_{\varsigma'}(t)=h_{\varsigma}(t)}$ for
$\varsigma>0$ (the cases $\varsigma\in\{0,\infty\}$ will be discussed afterwards) it is sufficient to note that for every $\varsigma,\varsigma'\in(0,\infty)$ with $\varsigma'<\varsigma$ and $t\geq 0$ we have by the very definition of $h_\varsigma$
\begin{equation}\label{e:hsigma vs hsigmaprimo}
 h_{\varsigma'}(t)\leq h_\varsigma(t)\leq 
 \left(\frac{\varsigma}{\varsigma'}\right)^2h_{\varsigma'}(t)\,.
\end{equation}
The latter formula also implies the validity of \eqref{e:hsigmacvxapprox} (the fact that
$h^{**}_\varsigma(t)<\infty$ is a consequence of \eqref{minorconlin} that will be proved below).

By definition $h_\varsigma$ is upper semicontinuous, thus its continuity in $0$ follows from its non negativity and $h_\varsigma(0)=0$. To establish the continuity of $h_\varsigma$ in $t>0$, it is then sufficient to show its left continuity, and actually that
$\displaystyle{\lim_{s\to t^-} h_\varsigma(s)\geq h_\varsigma(t)}$ by monotonicity.
Since by the very definition of $h_\varsigma$ for every $s\in(0,t)$ we have
\[
h_\varsigma(t)\leq
\left(\frac{t}{s}\right)^2h_\varsigma(s)\,,
\]
the conclusion then follows. Actually, the latter estimate implies that $h_\varsigma\in \mathrm{Lip}_{\mathrm{loc}}((0,\infty))$, with Lipschitz constant on any interval $[a,b]\subset(0,\infty)$ estimated by $\frac{2b}{a^2}h_\varsigma(b)$.

\noindent{\bf Step~2:} Proof of item (ii). 

Assume first that $\varphi'(0^+)\in[0,\infty)$.
Let $\varsigma>0$ and $t>0$ be fixed, and let $\eta_{\varsigma,t}>0$ be such that
$\varphi(\frac1{\eta_{\varsigma,t}})t^2=\frac{\varsigma^2}{4}\eta_{\varsigma,t}$, its existence  follows from the continuity and monotonicity of $\varphi$, then
\[
h_\varsigma(t)\leq(\varphi(\sfrac1{\eta_{\varsigma,t}})\eta_{\varsigma,t})^{\sfrac12}\,
\varsigma t\,.
\]
In particular, it easy to check that $\eta_{\varsigma,t}\to\infty$ as $t\to\infty$, so that
\[
\limsup_{t\to\infty}\frac{h_\varsigma^{**}(t)}t\leq
\limsup_{t\to\infty}\frac{h_\varsigma(t)}t\leq
(\varphi'(0^+))^{\sfrac12}\varsigma\,.
\]
If $\varphi'(0^+)=0$ we are done, otherwise to prove the opposite inequality
we assume first $\varphi$ to be bounded. Then $\sfrac{h_\varsigma(t)}{t}$ is upper bounded on $(0,\infty)$,
as by testing the definition of $h_\varsigma(t)$ with $\tau=t$ itself yields
\[
\frac{h_\varsigma(t)}{t}\leq \varphi(\sfrac1t)t+\frac{\varsigma^2}{4},
\]
and the claim follows being $\varphi$ continuous, bounded and differentiable in $t=0$.
Thus, $\tau_{\varsigma,t}\to\infty$ as $t\to\infty$, in turn implying
\[
\frac{h_\varsigma(t)}{t}=\varphi(\sfrac 1{\tau_{\varsigma,t}})t+\frac{\varsigma^2}{4}\frac{\tau_{\varsigma,t}}{t}
\geq(\varphi(\sfrac 1{\tau_{\varsigma,t}})\tau_{\varsigma,t})^{\sfrac12}\varsigma\,,
\]
and we may conclude
\[
\liminf_{t\to\infty}\frac{h_\varsigma(t)}t\geq
(\varphi'(0^+))^{\sfrac12}\varsigma\,.
\]
Hence, for every $\eps\in(0,(\varphi'(0^+))^{\sfrac12}\varsigma)$ there is  $\xi_\eps\geq 0$ such that  $\zeta(t):=0\vee(((\varphi'(0^+))^{\sfrac12}\varsigma-\eps) t-\xi_\eps)\leq h_\varsigma(t)$ for all $t\geq 0$.
In particular, $h_\varsigma^{**}\geq \zeta$, and thus
\[
\liminf_{t\to\infty}\frac{h_\varsigma^{**}(t)}t\geq
(\varphi'(0^+))^{\sfrac12}\varsigma-\eps\,,
\]
and \eqref{e:laaccabroo} for $h_\varsigma^{**}$ follows at once by letting $\eps\to 0$.
If $\varphi$ is not bounded, consider $\varphi\wedge j$ and $h_{\varsigma,j}$ the corresponding
function in \eqref{e:hsigma}. Then using \eqref{e:laaccabroo} for $h_{\varsigma,j}$ yields
${\displaystyle{\liminf_{t\to0}\textstyle{\frac{h_\varsigma^{**}(t)}{t}}}}\geq
{\displaystyle{\lim_{t\to0}\textstyle{\frac{h_{\varsigma,j}^{**}(t)}{t}}= (\varphi'(0^+))^{\sfrac12}\varsigma}}$,
and thus the conclusion.

If $\varphi'(0^+)=\infty$, there is a diverging sequence of integers $k_j\in\N$ such that
the connected component of $\{t\in(0,\infty):\,\varphi(t)>k_jt\}$ whose closure contains
the origin is bounded. Then consider the function $\varphi_j$ equal
to $k_j t$ on such a set and to $\varphi$ otherwise, and let $h_{\varsigma,j}$ be the
corresponding function in \eqref{e:hsigma}. Using \eqref{e:laaccabroo} for $h_{\varsigma,j}$ gives
${\displaystyle{\liminf_{t\to\infty}\textstyle{\frac{h_\varsigma^{**}(t)}{t}}\geq (k_j)^{\sfrac12}\varsigma}}$,
and thus \eqref{e:laaccabroo} for $h_\varsigma$ follows by letting $j\to\infty$.

The growth conditions in \eqref{minorconlin} are then a conseguence of the convexity of $h_\varsigma^{**}$,
and of the equality $h_\varsigma^{**}(0)=h_\varsigma(0)=0$.

We prove next \eqref{e:laaccabrodo}.
Assume first $\varphi(\infty)\in(0,\infty)$, and note that being $h_\varsigma\leq h_\infty$ we have
\[
\limsup_{t\to0}\frac{h_\varsigma^{**}(t)}{t^2}\leq
\limsup_{t\to0}\frac{h_\varsigma(t)}{t^2}\leq\varphi(\infty)\,.
\]
In addition, taking $\tau=t^3$ in the definition of $h_\varsigma(t)$, the monotonicity of $\varphi$ in
(Hp~$3$) yields
\[
\frac{h_\varsigma(t)}{t^2}=\varphi(\sfrac 1{\tau_{\varsigma,t}})+\frac{\varsigma^2}{4}\frac{\tau_{\varsigma,t}}{t^2}\leq
\varphi(\infty)+\frac{\varsigma^2}{4}t\,,
\]
from which we infer $\tau_{\varsigma,t}\leq\frac{4\varphi(\infty)}{\varsigma^2}t^2+t^3$. As, for every $\eps\in(0,\varphi(\infty))$ there is $\delta_\eps>0$ such that $\varphi(\sfrac 1\tau)\geq\varphi(\infty)-\eps$
if $\tau\in(0,\delta_\eps)$, there is $t_{\varsigma,\eps}>0$,
such that for $t\in(0,t_{\varsigma,\eps}]$
\[
h_\varsigma(t)\geq\varphi(\sfrac 1{\tau_{\varsigma,t}})t^2\geq (\varphi(\infty)-\eps)t^2\,.
\]
In particular, from this it is easy to infer the equality for the limit of $h_\varsigma$ in \eqref{e:laaccabrodo}.
Next, set $\eta_{\varsigma,\eps}:=\inf_{(t_{\varsigma,\eps},\infty)}\frac{h_\varsigma(t)-h_\varsigma(t_{\varsigma,\eps})}{t-t_{\varsigma,\eps}}$.
Thanks to \eqref{e:laaccabroo}, to the strict monotonicity and to the local Lipschitz continuity of $h_\varsigma$, we may assume $\eta_{\varsigma,\eps}>0$, up to substituting $t_{\varsigma,\eps}$ with a slightly smaller value $t_{\varsigma,\eps}'$ such that $h_\varsigma$ is differentiable on $t_{\varsigma,\eps}'$ and $h_\varsigma'(t_{\varsigma,\eps}')>0$. Therefore, $h_\varsigma\geq\psi_{\varsigma,\eps}$, where
\[
\psi_{\varsigma,\eps}(t):=\begin{cases}
(\varphi(\infty)-\eps)t^2 & \text{if $t\in[0,t_{\varsigma,\eps}]$}\\
(\varphi(\infty)-\eps)t_{\varsigma,\eps}^2 +
\eta_{\varsigma,\eps} (t-t_{\varsigma,\eps})
& \text{if $t>t_{\varsigma,\eps}$}
    \end{cases}\,,
\]
in turn implying $h^{**}_\varsigma\geq \psi_{\varsigma,\eps}^{**}$, where
\[
\psi_{\varsigma,\eps}^{**}(t)=\begin{cases}
(\varphi(\infty)-\eps)t^2 & \textrm{if $t\in[0,\frac{\eta_{\varsigma,\eps}}{2(\varphi(\infty)-\eps)}]$}\\
\eta_{\varsigma,\eps} t - \frac{\eta_{\varsigma,\eps}^2}{4(\varphi(\infty)-\eps)} & \textrm{if $t\in(\frac{\eta_{\varsigma,\eps}}{2(\varphi(\infty)-\eps)},\infty)$}
\end{cases}\,
\]
if $\eta_{\varsigma,\eps}<2(\varphi(\infty)-\eps)t_{\varsigma,\eps}$, and
$\psi_{\varsigma,\eps}^{**}=\psi_{\varsigma,\eps}$ otherwise.
In any case, from this it is easy to infer the equality for the limit of $h^{**}_\varsigma$ in \eqref{e:laaccabrodo}.
Finally, if $\varphi(\infty)=\infty$, consider $\varphi\wedge j$ and $h_{\varsigma,j}$ the corresponding function in \eqref{e:hsigma}. Then using \eqref{e:laaccabrodo} for $h_{\varsigma,j}$ yields
${\displaystyle{\liminf_{t\to0}\textstyle{\frac{h_\varsigma^{**}(t)}{t^2}}\geq j}}$, and the conclusion follows by letting
$j\to\infty$.

\noindent{\bf Step~3:} Proof of item (iii).

Let $t>0$, recalling that $\tau_{\varsigma,t}\geq0$ is a minimum point in the definition of $h_\varsigma(t)$,
from $h_\varsigma(t)\leq h_\infty(t)$ we conclude that $\tau_{\varsigma,t}\to 0^+$ as $\varsigma\to\infty$,
and thus $h_\varsigma(t)\to h_\infty(t)$ as $\varsigma\to\infty$.

To establish \eqref{e:hinftyapprox} for $h_\varsigma^{**}$, let $t>0$ (as $h_\varsigma^{**}(0)=h_\infty(0)=0$) and note that by definition
 $h^*_\varsigma(t)=-\inf\{h_\varsigma(s)-ts: s\in[0,\infty)\}$. Since $(h_\varsigma(\cdot)-t\cdot)_{\varsigma\geq0}$ is equicoercive for $\varsigma\geq\varsigma_t>t(\varphi'(0^+))^{-\sfrac12}$ thanks to \eqref{minorconlin}, the pointwise convergence of $h_\varsigma$ to $h_\infty$ and the fundamental theorem of $\Gamma$-convergence imply that $h^*_\varsigma(t)\to h^*_\infty(t)$
 as $\varsigma\to\infty$. Moreover, the family $(h^*_\varsigma)_{\varsigma\geq 0}$ is non-increasing with
 $h^*_\varsigma(t)\geq h^*_\infty(t)=\frac{t^2}{4}$, and thus $(h^*_\varsigma(\cdot)-s\cdot)_{\varsigma\geq0}$ is equicoercive, and again the pointwise convergence of $(h^*_\varsigma)_{\varsigma\geq 0}$ to $h^*_\infty$ implies $h^{**}_\varsigma(s)\to h^{**}_\infty(s)=h_\infty(s)$ for every $s\geq 0$.

\noindent{\bf Step~4:} Proof of item (iv).

The proof of \eqref{e:h0approx} is immediate by choosing $\tau=\sfrac1{\varsigma}$. Indeed, we have $h_\varsigma(t)\leq\varphi(\varsigma)t^2+\frac{\varsigma}2$, so that \eqref{e:h0approx} follows at once
by letting $\varsigma\to0^+$ by (Hp~$2$).
%
%
\end{proof}
\begin{remark}
We point out that in case $\varphi'(0^+)=\infty$ and $\varsigma\in(0,\infty)$, $h_{\varsigma}^{**}$
does not necessarily coincide with $\varphi(\infty)t^2$. Indeed, taking $\varphi(t)=t^\alpha\wedge 1$, $\alpha\in(0,1)$, an explicit calculation yields for every $t\geq 0$ the identity
$h_{\varsigma}^{**}(t)=(t^2\wedge(1+\alpha)(\frac{\varsigma^2}{4\alpha})^{\frac{\alpha}{1+\alpha}}t^{\frac{2}{1+\alpha}})^{**}$, which is sub-quadratic for large values of $t$.
\end{remark}

Next, we prove a truncation result following \cite[Lemma~4.4]{ContiFocardiIurlano2022}.
\begin{lemma}\label{l:suppenergy}
Assume (Hp~$1$), (Hp~$2$), and $\FailureS\in (0,\infty]$.
In addition, let $\varphi:[0,\infty)\to[0,\infty)$ be non-decreasing and $\varphi^{-1}(0)=\{0\}$.
Then, there exist two functions
$\zeta_1,\zeta_2:(0,1)\to (0,\infty)$ with
\begin{equation}\label{e:limittruncat}
\lim_{\delta \to 0}\zeta_1(\delta)=1 \qquad \textup{and} \qquad
\lim_{\delta \to 0}\zeta_2(\delta)=0
\end{equation}
satisfying the following property.
For every $\gamma\in(0,1\wedge \FailureS)$ there exists $\delta_\gamma\in(0,1)$ such that
for every $\delta\in(0,\delta_\gamma)$,
$(u,v)\in H^1(\Omega,\R\times[0,1])$, and $A\in \mathcal{A}(\Omega)$,
there is $s_\delta\in(1-\delta,1-\delta^2)$ such that $\chi_{\{v>s_\delta\}}\in BV(A)$,
$\tilde{u}_{\delta}:=u\,\chi_{\{v>s_\delta\}}\in SBV(A)$
(the previous quantities actually depend also on $\gamma$ and $A$) and
\begin{equation}\label{e:stimaprincipal}
 H_{\delta}(\tilde{u}_{\delta},A) \leq\Functeps(u,v,A)
\end{equation}
where $H_{\delta}:L^{1}(\Omega)\times \mathcal{A}(\Omega)\to [0,\infty]$
is defined for $\FailureS\in(0,\infty)$ by
\begin{equation}\label{e:Haccagrand sigma finito}
H_{\delta}(w,A)= \zeta_1(\delta)\int_A h_{\FailureS_\gamma}(|w'|) \dx + \zeta_2(\delta)\mathcal{H}^0(A\cap J_{w})
\end{equation}
if $w\in SBV(A)$, and $\infty$ otherwise
with $\FailureS_\gamma:=(\FailureS-\gamma)(1-\gamma)$; and for $\FailureS=\infty$ by
\begin{equation}\label{e:Haccagrand sigma infinito}
H_{\delta}(w,A)= \zeta_1(\delta)\int_A h_{\sfrac1\gamma-1}(|w'|) \dx
+ \zeta_2(\delta)\mathcal{H}^0(A\cap J_{w})
\end{equation}
if $w\in SBV(A)$, and $\infty$ otherwise.

Moreover, if $u_{\eps}\to u$ in $L^1(\Omega)$ and
$v_\eps\to 1$ in measure on $\Omega$, then $(\widetilde{u_{\eps}})_\delta \to u$ in $L^1(A)$ as $\eps\to 0$ for every $A\subseteq \mathcal{A}(\Omega)$ and $\delta \in (0,1)$.
\end{lemma}
\begin{proof}
With fixed $\gamma,\,\delta\in(0,1)$ and $(u,v)\in H^1(\Omega,\R\times[0,1])$
setting $A_\delta:=\{x\in A:\,v(x)>1-\delta\}$, we argue as follows:
\begin{align}\label{e:stima Functeps}
\Functeps&(u,v,A)\nonumber\\
&\geq\int_{A_\delta}\left(\varphi(\eps f^2(v))|u'|^2+(1-\delta)\frac{\PotDann(1-v)}{4\eps}\right)\dx
+\int_{A}\left(\delta\,\frac{\PotDann(1-v)}{4\eps}+\eps|v'|^2\right)\dx \nonumber\\
&\geq\int_{A_\delta}\left(\varphi(\eps f^2(v))|u'|^2+(1-\delta)\frac{\PotDann(1-v)}{4\eps}\right)\dx +\delta^{\sfrac12}\int_{A}|\Psi(v)'|\dx\,,
\end{align}
where $\Psi$ is the function defined in \eqref{e:Psi}.

We distinguish the two cases $\FailureS\in(0,\infty)$ and $\FailureS=\infty$.
We start with the former. Then, as $\hatf(1)=1$ and \eqref{e:FailureS def} holds,
for every $\gamma\in(0,\FailureS\wedge1)$ there is $\delta_\gamma\in(0,1)$ such that if $t\in(0,\delta_\gamma)$
\[
\hatf(1-t)\geq(1-\gamma)^2\,,\qquad
\frac{\PotDann(t)}{Q(t)}\geq (\FailureS-\gamma)^2\,.
\]
Therefore, the first summand in \eqref{e:stima Functeps} can be estimated as follows for $\delta\in(0,\delta_\gamma)$
by monotonicity of $\varphi$
\begin{align*}
\int_{A_\delta}&\left(\varphi(\eps f^2(v)) |u'|^2+(1-\delta)(\FailureS-\gamma)^2\frac{Q(1-v)}{4\eps}\right)\dx\\
&\geq(1-\delta)\int_{A_\delta}\left(
\varphi\left(\eps\frac{(1-\gamma)^2}{Q(1-v)}\right)
|u'|^2+(\FailureS-\gamma)^2\frac{Q(1-v)}{4\eps}\right)\dx\\
&\geq(1-\delta)\int_{A_\delta}h_{\FailureS_\gamma}(|u'|)\dx\,,
\end{align*}
where $\FailureS_\gamma=(\FailureS-\gamma)(1-\gamma)$.
Hence, \eqref{e:stima Functeps} yields that for $\delta\in(0,\delta_\gamma)$
\begin{align}\label{e:stima Functeps sigmabarra finito}
\Functeps(u,v,A)
\geq(1-\delta)\int_{A_\delta}h_{\FailureS_\gamma}(|u'|)\dx
+\delta^{\sfrac12}\int_{A}|\Psi'(v)|\dx\,.
\end{align}

If $\FailureS=\infty$, by \eqref{e:FailureS def} for every $\gamma\in(0,1)$
there is $\delta_\gamma\in(0,1)$ such that if $t\in(0,\delta_\gamma)$
\begin{equation}\label{e:lim def}
\frac{Q(t)}{\PotDann(t)}\leq \gamma^2\,, \quad \text{and}\quad \hatf(1-t)\geq(1-\gamma)^2\,,
\end{equation}
from which it follows that for $\delta\in(0,\delta_\gamma)$ we have by monotonicity of $\varphi$
\begin{align*}
\int_{A_\delta}&\left(\varphi(\eps f^2(v)) |u'|^2+(1-\delta)\frac{\PotDann(1-v)}{4\eps}\right)\dx\\
&\geq\int_{A_\delta}
\left(
\varphi\left(\frac{\eps(1-\gamma)^2}{Q(1-v)}\right)
|u'|^2+(1-\delta)\frac{\PotDann(1-v)}{4\eps}\right)\dx
\\
&\geq\int_{A_\delta}
\left(
\varphi\left(\frac{\eps}{\PotDann(1-v)}\frac{(1-\gamma)^2}{\gamma^2}\right)
|u'|^2+(1-\delta)\frac{\PotDann(1-v)}{4\eps}\right)\dx\\
&\geq(1-\delta)\int_{A_\delta}h_{\sfrac1\gamma-1}(|u'|)\dx\,.
\end{align*}
Therefore, we deduce that  for $\delta\in(0,\delta_\gamma)$
\begin{align}\label{e:stima Functeps sigmabarra infinito}
\Functeps(u,v,A)
\geq(1-\delta)\int_{A_\delta}h_{\sfrac1\gamma-1}(|u'|)\dx
+\delta^{\sfrac12}\int_{A}|\Psi'(v)|\dx\,.
\end{align}
To conclude we follow closely the argument in \cite[Lemma~4.4]{ContiFocardiIurlano2022}.
We observe that $\Psi$ is strictly increasing, and in particular $\Psi$ is bijective. By the coarea formula,
\begin{equation*}
\int_A |\Psi'(v)|\dx = \int_0^{\Psi(1)}
\calH^{0}(A\cap \partial^*\{\Psi(v)>t\}) \dt.
\end{equation*}
Therefore, there is $t_\delta\in (\Psi(1-\delta),\Psi(1-\delta^2))$ such that
\begin{equation*}
(\Psi(1-\delta^2)-\Psi(1-\delta))
\calH^{0}(A\cap \partial^*\{\Psi(v)>t_\delta\}) \le \int_A |\Psi'(v)|\dx.
\end{equation*}
We define $\tilde u:=u\chi_{\{\Psi(v)>t_\delta\}\cap A}$ (dropping the dependence on both $\delta$ and $A$ from $\tilde u$).
As $u\in L^\infty(\Omega)$, then $\tilde u\in SBV(A)$.
Being $h_\varsigma(0)=0$ for all $\varsigma$ we obtain either by
\eqref{e:stima Functeps sigmabarra finito} or by
\eqref{e:stima Functeps sigmabarra infinito}
\begin{align*}
\Functeps(u,v,A)\ge (1-\delta)\int_{A}h(|\tilde{u}'|)\dx
+ \delta^{\sfrac12}(\Psi(1-\delta^2)-\Psi(1-\delta))
\calH^{0}(A\cap J_{\tilde u})
\end{align*}
where $h=h_{\FailureS_\gamma}$ in the first case and $h=h_{\sfrac1\gamma-1}$
in the second case.
Defining $\zeta_1(\delta):=1-\delta$,
$\zeta_2(\delta):=\delta^{\sfrac12}(\Psi(1-\delta^2)-\Psi(1-\delta))$, and $s_\delta:=\Psi^{-1}(t_\delta)$
we deduce \eqref{e:stimaprincipal} and \eqref{e:limittruncat}.

We also remark that $\|\tilde u-u\|_{L^1(A)}
\leq\|u\|_{L^1(\{v\le \Psi^{-1}(t_\delta)\})}$, hence, if the sequence $u_\eps$ is equi-integrable and $v_\eps\to1$ in
measure on $A$, we obtain that $u_\eps-\tilde u_\eps\to0$ in $L^1(A)$ for every $\delta\in(0,1)$.
\end{proof}

Next, we resume a partial relaxation result established along the proof of \cite[Proposition~4.2]{ContiFocardiIurlano2022} in the form needed in this paper.
\begin{proposition}\label{p:CFI_sci}
Let $\zeta\in(0,\infty)$, and $\phi:\R\to[0,\infty)$ be convex such that
\begin{equation*}
0\le \phi(t)\le c(1+|t|) \quad\textup{for all } t\in\R\,.
\end{equation*}
Define $\Phi:L^1(\Omega)\times\calA(\Omega)\to[0,\infty)$ by
\begin{equation*}
\Phi(w,A):=\int_A \phi(|w'|) \dx + \zeta\,\mathcal{H}^0(A\cap J_{w})
\end{equation*}
if $w\in SBV(\Omega)$, and $\infty$ otherwise.

Then, for any $w_j,w\in BV(\Omega)$, with $w_j\to w$ in $L^1(A)$, we have
\begin{equation*}
\liminf_j\Phi(w_j,A)\geq\int_A\phi(|w'|)\dx+\phi^\infty|D^cw|(A)\,,
\end{equation*}
where
\[
\phi^\infty:=\lim_{t\to\infty}\frac{\phi(t)}{t}\in[0,c]\,.
\]
\end{proposition}

\subsection{Properties of the surface energy densities}

We collect here some useful properties of the function $g$.
With this aim, for every $s\in [0,\infty)$ we introduce the notation
\begin{align}\label{e:glispaziu}
  \calulu_{s}(0,T)=\{(\gamma,\beta) \in H^1((0,T),\R^2):\,&
  \gamma(0)=0\,,\gamma(T)=s\,, 
  0\leq \beta \leq 1\,, \beta(0)=\beta(T)= 1\}\,,
\end{align}
for every $T>0$, with the following convention
$\calulu_{s}:=\calulu_{s}(0,1)$.

\subsubsection{Case \texorpdfstring{$\FailureS\in(0,\infty)$}{...}}

In case $\FailureS\in(0,\infty)$, we state without proof some results on $g$
whose proofs can be obtained following word-by-word the arguments used in
the case $\PotDann(t)=Q(t)=t^2$ in \cite{ContiFocardiIurlano2016,BCI,BI23}
(in the notation of the last two papers $\hatf(t)=f_1^2(1-t)$).
We recall that $\Psi$ is the function defined in \eqref{e:Psi}.

\begin{proposition}\label{p:lepropdig}
Under the assumptions of Theorem~\ref{t:finale} with $\FailureS\in(0,\infty)$, the function $g$ defined in \eqref{e:lagiyo} enjoys the following properties:
\begin{itemize}
\item[(i)] $g(0)=0$ and $g$ is subadditive;

\item[(ii)] $g$ is non-decreasing, $g(s)\leq (\varphi'(0^+))^{\sfrac12}\,\FailureS s\wedge 2\Psi(1)$, $g$ is
Lipschitz continuous with Lipschitz constant equal to $(\varphi'(0^+))^{\sfrac12}\,\FailureS$;

\item[(iii)] the ensuing limit exists and
\begin{equation}\label{e:Lafigrandeinfi}
\lim_{s\to \infty} g(s)=2\Psi(1);
\end{equation}

\item[(iv)] the ensuing limit exists and
\begin{equation}\label{e:strenght}
\lim_{s\to 0}\frac{g(s)}{s}=(\varphi'(0^+))^{\sfrac12}\,\FailureS.
\end{equation}

\item[(v)]
the following alternative representation for $g$ holds
\begin{align}\label{e:lagicappucciobrodo}
 g(s)&=\inf_{T>0}\inf_{(\gamma,\beta)\in \mathfrak{U}_s(0,T)}
 \int_0^T\left(\varphi'(0^+)f^2(\beta)|\gamma'|^2 + \frac{\PotDann(1-\beta)}{4}+|\beta'|^2  \right)\dx
\end{align}

\item[(vi)] For all $\eta \in [0,1]$ and $s\in [0,\infty)$
\begin{equation}\label{e:legetayoobro}
 0\leq g(s)-g_\eta(s)\leq 2\lambda_{0}(\eta)\eta
\end{equation}
where $\lambda_{0}(s)=\max_{t\in [0,s]}\PotDann^{\sfrac12}(t)$, and
$g_\eta:[0,\infty)\to [0,\infty)$ is defined by
\begin{equation}\label{e:legieta}
    g_\eta(s):=\inf_{(\gamma,\beta)\in \calulu^{\eta}_s}\int_0^1
\Big(\PotDann(1-\beta)\big((\varphi'(0^+)f^2(\beta)|\gamma'|^2+|\beta'|^2\big)\Big)^{\sfrac12}
    \dx
\end{equation}
where 
\begin{align}\label{e:glispaziueta}
  \calulu^{\eta}_s
  :=\{(\gamma,\beta) \in H^1((0,1),\R^2):\,&
  \gamma(0)=0\,,\gamma(1)=s, \; 
  0\leq \beta \leq 1,\; \textup{and}\; \beta(0),\beta(1)\geq 1-\eta\}
\end{align}
\end{itemize}
\end{proposition}

\subsubsection{Case \texorpdfstring{$\FailureS=\infty$}{...}}
We turn next to establish the structural properties of the surface energy density $g$ in case $\FailureS=\infty$. The proof somewhat follows that of \cite[Propositions~4.1 and 7.3]{ContiFocardiIurlano2016} to which we refer in case it is an immediate adaptation of the latter. We highlight only the main changes.
\begin{proposition}\label{p:lepropdig infty}
Under the assumptions of Theorem~\ref{t:finale} with $\FailureS=\infty$,
the function $g$ defined in \eqref{e:lagiyo} enjoys the following properties:
\begin{itemize}
\item[(i)] $g(0)=0$,  $g$ is non-decreasing, and subadditive;
\item[(ii)] $g\in C^0([0,\infty))$, $0\leq g(s)\leq \widehat{g}(s)$ for every $s\geq 0$, where
\begin{equation}\label{e:g zero def}
    \widehat{g}(s):=\inf_{x\in(0,1]}\left\{2(\Psi(1)-\Psi(1-x))+\Big(\varphi'(0^+)\hatf(1-x)\frac{\PotDann(x)}{Q(x)}\Big)^{\sfrac12}\,s\right\}\,,
\end{equation}
$\widehat{g}\in C^0([0,\infty))$ is concave, non-decreasing, $\widehat{g}(s)\leq 2\Psi(1)$, and
\begin{equation}\label{e:g zero behaviour in 0}
\lim_{s\to 0^+}\frac{\widehat{g}(s)}s=\infty.
\end{equation}

\item[(iii)]
\begin{equation}\label{e:g behaviour in 0}
2^{-\sfrac12}\leq\liminf_{s\to 0^+}\frac{g(s)}{\widehat{g}(s)}\leq
\limsup_{s\to 0^+}\frac{g(s)}{\widehat{g}(s)}\leq 1\,,
\end{equation}

\item[(iv)]
\begin{equation*}
\lim_{s\to \infty} g(s)=2\Psi(1);
\end{equation*}

\item[(v)]
the alternative representation for $g$ in \eqref{e:lagicappucciobrodo} holds.
\end{itemize}
\end{proposition}
\begin{proof}
For the proofs of items (i), (iv) and
(v) we refer to \cite[Proposition~4.1]{ContiFocardiIurlano2016}.
To prove (ii), let $s\in(0,1)$ and consider $\gamma=0$ on $[0,\sfrac13]$, $\gamma=s$ on $[\sfrac23,1]$ and the linear interpolation between $0$ and $s$ on $[\sfrac13,\sfrac23]$, while
$\beta=1-\delta$ on $[\sfrac13,\sfrac23]$ and the linear interpolation between $\delta$ and $1$ on $[0,\sfrac13]\cup[\sfrac23,1]$, where $\delta\in(0,1]$ is arbitrary. 
In particular, $(\gamma,\beta)\in \mathfrak{U}_s$ and a simple calculation gives
\begin{equation*}
    g(s)\leq 2(\Psi(1)-\Psi(1-\delta))+\left(\varphi'(0^+)\hatf(1-\delta)\frac{\PotDann(\delta)}{Q(\delta)}\right)^{\sfrac12}\,s \,.
\end{equation*}
Therefore, by \eqref{e:g zero def} it holds that for every $s\geq 0$
\begin{equation}\label{e:g zero}
    g(s)\leq \widehat{g}(s)\,.
\end{equation}
Note that by its very definition $\widehat{g}$ is concave, non-decreasing, $\widehat{g}(s)> \widehat{g}(0)=0$, and
$\widehat{g}(s)\leq 2\Psi(1)$ for every $s> 0$. Therefore, $\widehat{g}$ is continuous.
Moreover, by assumptions (Hp~$1$)-(Hp~$4$) and as $\FailureS=\infty$, it is easy to infer
that for $s>0$ at least a minimizer exists for the problem defining $\widehat{g}(s)$.
Clearly, all minimizers are strictly positive if $s>0$.
Denote by $\zeta(s)$ the smallest minimizers, which exists by continuity of $\hatf$, $\PotDann$, $Q$, and $\Psi$.
Then $\zeta(s)>0$ if $s>0$, and as $\widehat{g}(s)\geq 2(\Psi(1)-\Psi(1-\zeta(s)))$, we have $\zeta(s)\to 0^+$ as $s\to0^+$.
Thus, to establish \eqref{e:g zero behaviour in 0} it is sufficient to take into account that
$\FailureS=\infty$ and to note that
\[
\frac{\widehat{g}(s)}s
\geq\Big(\varphi'(0^+)\hatf(1-\zeta(s))\frac{\PotDann(\zeta(s))}{Q(\zeta(s))}\Big)^{\sfrac12}\,.
\]
Having fixed $s_1\,,s_2\in[0,1]$, by the monotonicity and subadditivity
of $g$, \eqref{e:g zero} yields that $|g(s_2)-g(s_1)|\leq g(|s_2-s_1|)\leq \widehat{g}(|s_2-s_1|)$.
The continuity of $g$ then follows.

We establish next item (iii). First, note that by \eqref{e:g zero}
it immediately follows the inequality on the right hand side of \eqref{e:g behaviour in 0}.
On the other hand, if $\{s_j\}_{j\in \N}$ is an infinitesimal sequence realizing the inferior limit in
\eqref{e:g behaviour in 0},  let $\lambda_j=o(\widehat{g}(s_j))\geq 0$ as $j\to\infty$, and
$(\gamma_j,\beta_j)\in \calulu_{s_j}$ (cf. \eqref{e:glispaziu}) be competitors such that
\begin{equation*}
    \int_0^1 \Big(\PotDann(1-\beta_j)\big(\varphi'(0^+)f^2(\beta_j)|\gamma_j'|^2+|\beta_j'|^2\big)\Big)^{\sfrac12}\dx \leq g(s_j)+\lambda_j\,,
\end{equation*}
then defining $\xi_j(x):=\big(\varphi'(0^+)\hatf(\beta_j(x))\frac{\PotDann(1-\beta_j(x))}{Q(1-\beta_j(x))}\big)^{\sfrac12}$,
and using the concavity of the square root we have
\begin{align*}
g(s_j)+\lambda_j&\geq 2^{-\sfrac12}
\int_0^1\left(
\xi_j|\gamma_j'|+|(\Psi(\beta_j))'|\right)\dx\notag\\
&\geq 2^{-\sfrac12}\Big(
\xi_j(x_j)\,s_j
+2(\Psi(1)-\min_{[0,1]}\Psi(\beta_j))\Big)\notag\\
&\geq2^{-\sfrac12}\left(
\xi_j(x_j)\,s_j+2(\Psi(1)-\Psi(\beta_j(x_j))\right)\geq
2^{-\sfrac12}\widehat{g}(s_j)\,,
\end{align*}
where $x_j\in(0,1)$ denotes an absolute minimum
point of $\xi_j\in C^0((0,1))$ (note that $\xi_j\to\infty$ both as $x\to0^+$
and as $x\to1^-$ being $\beta_j\in\calulu_{s_j}$).
Then the first inequality in \eqref{e:g behaviour in 0} follows at once being $\lambda_j=o(\widehat{g}(s_j))$ as $j\to\infty$.
\end{proof}

\section{\texorpdfstring{$\Gamma$}{...}-convergence and compactness}\label{s:gammaconv compactness}

In this section we address the $\Gamma$-convergence and compactness properties of the family $\{\Functeps\}_{\eps>0}$.
We distinguish the cases $\FailureS\in(0,\infty)$,
$\FailureS=\infty$.

\subsection{Case \texorpdfstring{$\FailureS\in(0,\infty)$}{...}}

We begin with establishing compactness and identifying the domain of the eventual $\Gamma$-limit.
\begin{theorem}[Compactness]\label{t:cptness sigma finito}
Under the assumptions of Theorem~\ref{t:finale} with $\FailureS\in(0,\infty)$,
let $\{(u_{\eps},v_{\eps})\}_{\eps>0} \in L^1(\Omega,\R^2)$ be such that
\begin{equation}\label{e:supremumnorme sigma finito}
    \sup_{\eps >0} (\Functeps(u_{\eps},v_{\eps})+\|u_{\eps}\|_{L^1(\Omega)})<\infty\,.
\end{equation}
Then there are 
$\{\eps_k\}_{k\in \N}$ and $u\in L^1\cap GBV(\Omega)$ such that
$(u_{\eps_k},v_{\eps_k})\to (u,1)$ $\calL^1$-a.e. on $\Omega$.
If, moreover, 
$(u_\eps)_\eps$ is equi-integrable then $u_{\eps_k}\to u$ in $L^1(\Omega)$.
\end{theorem}
\begin{proof}
Let $C>0$ be the supremum on the left hand side of \eqref{e:supremumnorme sigma finito}, then
\[
\int_\Omega\PotDann(1-v_\eps)\dx\leq C\eps\,,
\]
so that $v_\eps\to 1$ in measure on $\Omega$, and moreover
$(u_{\eps},v_{\eps})\in H^1(\Omega,\R\times[0,1])$.

We show next that there is a subsequence of $(u_\eps)_\eps$ converging in measure on $\Omega$.
With this aim, fix $\gamma\in(0,\FailureS\wedge 1)$, then Lemma~\ref{l:suppenergy} provides $\delta_\gamma>0$ and functions $\tilde{u}_{\eps,\delta}:=(\widetilde{u_\eps})_{\delta}\in SBV(\Omega)$ such that
$H_\delta(\tilde u_{\eps,\delta})\leq \Functeps(u_\eps,v_\eps)$ for every $\delta\in(0,\delta_\gamma)$, with bulk density $h_{\FailureS_\gamma}$ where $\FailureS_\gamma=(\FailureS-\gamma)(1-\gamma)$
(cf. \eqref{e:Haccagrand sigma finito}).
Therefore, by $\eqref{minorconlin}$ we deduce that for some constant $C_0$
depending on $\FailureS$, $\gamma$ and $\delta$ we have
\[
\sup_\eps
\int_\Omega|\tilde{u}_{\eps,\delta}'|\dx+\calH^0(J_{\tilde{u}_{\eps,\delta}})
\leq C_0\,.
\]
Considering the truncated functions $\tilde u_{\eps,\delta}^M
\in SBV(\Omega)$ for $M\in\N$ (cf. \eqref{e:truncations}), \eqref{e:supremumnorme sigma finito} and the previous estimate yield that $\sup_{\eps>0}\|\tilde{u}^M_{\eps,\delta}\|_{BV(\Omega)}\leq C_M<\infty$. The $BV$ compactness Theorem and an elementry diagonal argument imply the existence of a subsequence $\eps_k$ (independent from $M$, but depending on $\delta$) and of $u^M\in BV(\Omega)$ such that $\tilde{u}_{\eps_{k},\delta}^M\to u^M$ in $L^1(\Omega)$ for every $M\in\N$.

We recall that in the proof of Lemma~\ref{l:suppenergy} we have set $\tilde{u}_{\eps,\delta}=u_\eps\chi_{\{ \Psi(v_\eps)>t_{\eps,\delta}\}}$ for some $t_{\eps,\delta}\in(\Psi(1-\delta),\Psi(1-\delta^2))$, and thus  $\calL^1(\{\Psi(v_\eps)\leq t_{\eps,\delta}\})\to 0$ as $\eps\to0$ as $v_\eps\to 1$ in measure on $\Omega$.
Thus, we infer the $(u_{\eps_k}^M)_k$ convergence in measure on $\Omega$ to $u^M$, i.e.
\begin{equation}\label{e:conv measure uepsk}
\limsup_k\calL^1(\{|\tilde{u}_{\eps_k,\delta}^M-u_{\eps_k}^M|\geq\eta\})
 \leq\limsup_k\calL^1(\{\Psi(v_\eps)< t_{\eps,\delta}\})=0\,.
\end{equation}

It is easy to check that $u^{M+1}=u^M$ if $|u^{M+1}|\leq M$, therefore defining $u:=\sup_{M\in\N}u^M$
we conclude that $u\in GBV(\Omega)$. In addition, we have
$u\in L^1(\Omega)$ as
\begin{align}\label{e:u norm}
\|u\|_{L^1(\Omega)}&\leq
\liminf_M\|u^M\|_{L^1(\Omega)}\nonumber
\\&\leq\liminf_M(\liminf_k
\|u_{\eps_{k}}^M\|_{L^1(\Omega)})\leq
\liminf_k
\|u_{\eps_{k}}\|_{L^1(\Omega)}\leq C\,.
\end{align}
Finally, for every $\eta>0$ we have that
\begin{align*}
\calL^1&(|u-u_{\eps_{k}}|\geq \eta)\leq
\calL^1(|u|> M)+\calL^1(|u^M-u_{\eps_k}^M|\geq \sfrac\eta3)\\&
+\calL^1(|u_{\eps_k}^M-u_{\eps_{k}}|\geq\sfrac\eta3)
\leq\calL^1(|u^M-u_{\eps_k}^M|\geq \sfrac\eta3)
+2\frac CM\,,
\end{align*}
where in the last inequality we have used
\eqref{e:supremumnorme sigma finito} and \eqref{e:u norm}.
From this we conclude that $(u_{\eps_{k}})_k$
converges in measure on $\Omega$ to $u$ by letting first $k\to\infty$ and then $M\to\infty$.

The claimed $\calL^1$-a.e. convergence on $\Omega$ holds up to extracting a further subsequence. Finally, if $(u_\eps)_\eps$ is equi-integrable,
by taking into account \eqref{e:conv measure uepsk}
we estimate as follows
\begin{align*}
\|u-u_{\eps_k}\|_{L^1(\Omega)}&\leq\int_{\{|u|>M\}} |u|\dx+\|u^M-u_{\eps_k}^M\|_{L^1(\Omega)}
+\|u_{\eps_k}^M-u_{\eps_k}\|_{L^1(\Omega)}\\
&\leq \int_{\{|u|>M\}} |u|\dx+\|u^M-\tilde{u}^M_{\eps_k,\delta}\|_{L^1(\Omega)}\\
&+\int_{\{\Psi(v_{\eps_k})\leq t_{\eps_k,\delta}\}}|u_{\eps_k}|\dx
+\int_{\{|u_{\eps_k}|>M\}} |u_{\eps_k}|\dx\,.
\end{align*}
The 
conclusion then follows by letting first $k\to\infty$ and then $M\to\infty$.
\end{proof}

Next, we show the lower bound inequality for the diffuse part, for which we follow \cite[Proposition~4.2]{ContiFocardiIurlano2022}.
\begin{proposition}\label{p:diffuseliminf sigmabarra finito}
Under the assumptions of Theorem~\ref{t:finale} with $\FailureS\in(0,\infty)$,
for every $(u_{\eps},v_{\eps}),\,(u,v)\in L^1(\Omega,\R^2)$ with
$(u_{\eps},v_{\eps})\to (u,v)$ in $L^1(\Omega,\R^2)$, we have
\begin{equation}\label{e:liminf diffuse sigmabarra finito}
    \int_A \hsigmabarra^{**}(|u'|)\dx + (\varphi'(0^+))^{\sfrac12}\FailureS |D^cu|(A) \leq \liminf_{\eps \to 0} \Functeps(u_{\eps},v_{\eps},A)\,,
\end{equation}
for every $A\in\calA(\Omega)$.
\end{proposition}
\begin{proof}
By superadditivity of the inferior limit it suffices to assume that $A\in\calA(\Omega)$ is an interval.
Moreover, we can suppose the inferior limit on the right hand side of
\eqref{e:liminf diffuse sigmabarra finito} to be finite, otherwise the claim
is obvious.
Therefore, $(u_{\eps},v_{\eps})\in H^1(A,\R\times[0,1])$ for $\eps>0$ sufficiently small, $u\in GBV(A)$
and $v=1$ $\calL^1$-a.e. on $A$ by Theorem~\ref{t:cptness sigma finito} applied on $A$ in place of $\Omega$.
Let $M\in\N$ and $\gamma\in(0,1)$, using the notation and the results in Lemma~\ref{l:suppenergy}
we find $\delta_{\gamma,M}\in(0,1)$ such that
for every $\delta\in(0,\delta_{\gamma,M})$ if
$\tilde{u}_{\eps,\delta}^M:=(\widetilde{u_\eps^M})_\delta\in SBV(A)$
(for every $\eps$ small enough) such that
\[
H_\delta(\tilde u_{\eps,\delta}^M,A)\leq \Functeps(u_\eps^M,v_\eps,A)
\leq\Functeps(u_\eps,v_\eps,A)\,.
\]
Thus, by taking the inferior limit as $\eps\to0$, in view of Proposition~\ref{p:CFI_sci} we conclude that
for $\FailureS_\gamma=(\FailureS-\gamma)(1-\gamma)$
\begin{equation}
    \zeta_1(\delta)\int_A h_{\FailureS_\gamma}^{**}(|(u^M)'|)\dx + \zeta_1(\delta)(\varphi'(0^+))^{\sfrac12}\FailureS_\gamma|D^cu^M|(A) \leq \liminf_{\eps \to 0} \Functeps(u_{\eps},v_{\eps},A)\,.
\end{equation}
By letting first $\delta\to0$, $\gamma\to0$ and then $M\to\infty$, we conclude \eqref{e:liminf diffuse sigmabarra finito} in view of \eqref{e:limittruncat}, \eqref{e:hsigmacvxapprox} and Beppo Levi's theorem.
\end{proof}

We establish next the lower estimate for the surface part.
\begin{proposition}\label{p:lowerboundBVsfc}
Under the assumptions of Theorem~\ref{t:finale} with $\FailureS\in(0,\infty)$,
for every $(u_{\eps},v_{\eps}),\,(u,v)\in L^1(\Omega,\R^2)$ with
$(u_{\eps},v_{\eps})\to (u,v)$ in $L^1(\Omega,\R^2)$, we have
\begin{equation}\label{lbjump}
\int_{J_u{\cap A}}g(|[u]|)\dd\calH^{0}  \le
\liminf_{\eps\to0}\Functeps(u_\eps,v_\eps,A)\,,
\end{equation}
for every $A\in\calA(\Omega)$,
where $g$ is the function defined in \eqref{e:lagiyo}.
\end{proposition}
\begin{proof}
By superadditivity of the inferior limit it suffices to assume that $A\in\calA(\Omega)$ is an interval.
Moreover, we can suppose the inferior limit on the right hand side of \eqref{lbjump} to be finite, otherwise the claim is obvious.
Therefore, $(u_{\eps},v_{\eps})\in H^1(A,\R\times[0,1])$ for $\eps>0$ sufficiently small, $u\in GBV(A)$
and $v=1$ $\calL^1$-a.e. on $A$ by Theorem~\ref{t:cptness sigma finito}.

We claim that it is sufficient to show that if $u\in BV(A)$ we have
for every $x_0\in J_u\cap A$
\begin{equation}\label{e:blowup jump x}
g(|[u](x_0)|)\leq\liminf_{r\to 0}
\liminf_{\eps\to 0}\Functeps(u_\eps,v_\eps,I_r(x)),
\end{equation}
where $I_r(x)=(x-\sfrac r2,x+\sfrac r2)$.
Indeed, given \eqref{e:blowup jump x} for granted, if $u\in BV(A)$
we can find a nested sequence of finite sets $\{K_m\}_{m\in N}$ such that $\cup_{m\in \N}K_m=J_u\cap A$. Then, for every $m\in \N$, there is $r_m>0$ such that $\{I_r(x)\}_{x\in K_m}$ are disjoint and contained in $A$ for all $r\in (0,r_m)$. In particular, by the superadditivity of the inferior limit operator and by \eqref{e:blowup jump x}, for every $m\in \N$ we get
\begin{align*}
\int_{K_m}&g(|[u](x)|)\dd\calH^0(x) \leq
\sum_{x\in K_m} \liminf_{r\to 0}\liminf_{\eps\to 0} \Functeps(u_{\eps},v_{\eps},I_r(x))\\
&\leq \liminf_{r\to 0}\sum_{x\in K_m} \liminf_{\eps \to 0}\Functeps(u_{\eps},v_{\eps},I_r(x)) \leq\liminf_{\eps \to 0}\Functeps(u_{\eps},v_{\eps},A)\,.
\end{align*}
Taking the limit for $m\to \infty$ we conclude \eqref{lbjump} for $u\in BV(A)$.

Moreover, if $u\in GBV\setminus BV(A)$ then $J_u=\cup_{M\in\N}J_{u^M}$
and $[u^M](x_0)=[u](x_0)$ for $M\in\N$ sufficiently large. Noting that $\Functeps(u_{\eps}^M,v_{\eps},A)\leq\Functeps(u_{\eps},v_{\eps},A)$, we conclude
\eqref{lbjump} for $u$ thanks to \eqref{lbjump} for $u^M$ and Beppo Levi's theorem.
\medskip

To establish \eqref{e:blowup jump x} for $u\in BV(A)$ we argue by blow up.
We can restrict to a subsequence $\mathcal{F}_{\eps_k}(u_{\eps_k},v_{\eps_k})$ such that
$\displaystyle{\liminf_{\eps \to 0}
\Functeps(u_\eps,v_\eps,A)=\lim_{k\to \infty}
\mathcal{F}_{\eps_k}(u_k,v_k) 
<\infty}$,
where we have set $(u_k,v_k):=(u_{\eps_k},v_{\eps_k})$.
Next, consider an infinitesimal sequence of radii
$\{r_j\}_{j \in \mathbb{N}}$ such that
   \begin{equation*}
       \liminf_{r\to 0}\liminf_{k\to \infty} \Functepsk(u_k,v_k,I_r(x_0)) = \lim_{j\to \infty}\liminf_{k\to \infty} \Functepsk(u_k,v_k,I_{r_j}(x_0))\,,
   \end{equation*}
and select a subsequence $\{k_j\}_{j\in \mathbb{N}}$ such that $\eta_j:=\sfrac{\eps_{k_j}}{r_j } < \sfrac1j$,
\begin{equation*}
    | \Functepskj(u_{k_j},v_{k_j},I_{r_j}(x_0))- \liminf_{k\to \infty} \Functepsk(u_k,v_k, I_{r_j}(x_0))  |< \sfrac{1}{j}
\end{equation*}
and
\begin{equation}\label{e:convl1}
\|v_{k_j}-1\|_{L^1(\Omega)}+\|u_{k_j}-u\|_{L^1(\Omega)}<\sfrac{r_j}{j}\,
\end{equation}
for all $j\in \mathbb{N}$. Therefore, we have
\begin{equation}\label{e:oci}
    \liminf_{r\to 0}\liminf_{k\to \infty} \Functepsk(u_k,v_k,I_r(x_0))=\lim_{j\to \infty} \Functepskj(u_{k_j},v_{k_j},I_{r_j}(x_0)).
\end{equation}
For every $j\in \mathbb{N}$ define the pair $(\tilde{u}_j,\tilde{v}_j)\in H^1(I_1,\R\times[0,1])$ by
$\tilde{u}_j(x):=u_{k_j}(x_0+r_jx)$ and
$\tilde{v}_j(x):=v_{k_j}(x_0+r_jx)$ for all $x\in I_1$.
A change of variable yields that
 \begin{equation}\label{e:first reduction}
     \Functepskj(u_{k_j},v_{k_j},I_{r_j}(x_0))=
     \calG_j(\tilde{u}_j,\tilde{v}_j,I_1)
 \end{equation}
 where 
 $\calG_j:L^1(I_1,\R^2\to[0,\infty]$ is defined for every $(u,v)\in H^1(I_1,\R\times[0,1])$ by
   \begin{equation}\label{e:Gj}
\calG_j(u,v) = \int_{I_1}\left( 
\frac{\eta_j}{\eps_{k_j}}\varphi
(\eps_{k_j}f^2(v))|u'|^2 + \frac{\PotDann(1-v)}{4\eta_j} +\eta_j|v'|^2\right) \dx\,,
 \end{equation}
 and $\infty$ otherwise.
In addition, changing variables it is straightforward to check that inequality \eqref{e:convl1} implies that
 \begin{align}\label{e:conv vjtilde}
     \limsup_{j\to \infty}&(\|\tilde{u}_j - u_0\|_{L^1(I_1)} +\|\tilde{v}_{j}-1\|_{L^1(I_1)})\nonumber \\
     &\leq \limsup_{j\to \infty} r_j^{-1}(\|u_{k_j}-u\|_{L^1(\Omega)}+\|v_{k_j}-1\|_{L^1(\Omega)})=0\,,
 \end{align}
where $u_0\in BV(I_1)$ is given by
$u_0(x):=u(x_0^-)\chi_{(-\frac12,0)} +u(x_0^+)\chi_{[0,\frac12)}$, and
we have used that
$\displaystyle\lim_{r\to 0}\|u(x_0+rx)-u_0(x)\|_{L^1(I_1)} =0$.

With fixed $\delta \in (0,1)$, for all $\eps>0$ set
$\delta_{\eps}:=\sup\left\{t\in [0,1):\,
\eps f^2(t) \leq\delta\right\}$.
Recalling that $f(0)=0$ and $f(t)\to\infty$ as $t\to 1^-$, $\delta_\eps\in(0,1)$ is actually a maximum
by the continuity of $f$ on $[0,1)$. Moreover, we have for all $t\in (\delta_{\eps},1)$
\begin{equation}\label{e:bounddeltaciro}
\delta< \eps f^2(t)\,
\end{equation}
and actually
\begin{equation}\label{e:deltaeps}
\eps f^2(\delta_\eps)=\delta\,.
\end{equation}
In particular, the latter equation implies $\delta_\eps\to 1$ as $\eps\to 0$.

We show next that
\begin{equation}\label{e:bounddeltaciro bis}
\eps f^2(t)\leq\delta\,.
\end{equation}
for all $t\in [0,\delta_{\eps}]$ if $\eps$ is sufficiently small. Indeed, recalling that $f$ is non-decreasing
on $[\gamma,1)$ by (Hp~$1$), define $M:=\max_{[0,\gamma]}f$. Let $\eps_0>0$ be such that
$\delta_\eps\in(\gamma,1)$ and $M^2<\delta/\eps$ for all $\eps\in(0,\eps_0)$. Then, we estimate as follows:
$\max_{[0,\delta_\eps]}f^2=M^2\vee\max_{[\gamma,\delta_\eps]}f^2=M^2\vee f^2(\delta_\eps)=\frac\delta\eps$.

Then, for every $j\in \mathbb{N}$ define $\hat{v}_j=\tilde{v}_j\land \delta_j$, where $\delta_j:=\delta_{\eps_{k_j}}$. We have
\begin{align}\label{e:primaeq}
\calG_j(\tilde{u}_j,\hat{v}_j,I_1)&=
\calG_j(\tilde{u}_j,\tilde{v}_j,\{\tilde{v}_j<\delta_{j}\})\nonumber\\
&+\int_{\{\tilde{v}_j\geq \delta_{j}\}}
\frac{\eta_j}{\eps_{k_j}}\varphi(\delta)
|\tilde{u}'_{j}|^2 \dx +\frac{\PotDann(1-\delta_{j})}{4\eta_j} \mathcal{L}^1(\{\tilde{v}_j\geq \delta_{j}\})\,.
\end{align}
We analyze the first term in the last line of \eqref{e:primaeq}: we employ \eqref{e:bounddeltaciro} 
to infer
 \begin{align}\label{e:secondeq}
\int_{\{\tilde{v}_j\geq \delta_{j}\}}
\frac{\eta_j}{\eps_{k_j}}\varphi(\delta)|\tilde{u}'_{j}|^2\dx
\leq \int_{\{\tilde{v}_j\geq \delta_{j}\}}\frac{\eta_j}{\eps_{k_j}}\varphi
(\eps_{k_j} f^2(\tilde{v}_j))|\tilde{u}'_{j}|^2 \dx\leq \calG_j(\tilde{u}_j,\tilde{v}_j,\{\tilde{v}_j\geq \delta_{j}\}) .
 \end{align}
 Instead, for what the second term in the last line of \eqref{e:primaeq} is concerned, being $\FailureS$ finite,
 the definition of $\eta_j$ and the identity in \eqref{e:deltaeps}
 imply that
\begin{equation}\label{e:cicolella}
     \limsup_{j\to \infty}\frac{\PotDann(1-\delta_{j})}{\eta_j} \mathcal{L}^1(\{\tilde{v}_j\geq \delta_{j}\})\leq
     \FailureS^2\cdot\lim_{j\to \infty}\frac{Q(1-\delta_{j})}{\eps_{k_j}}r_j
=\FailureS^2\cdot\lim_{j\to \infty}\frac{\hatf(\delta_j)}{\delta} r_j=0\,.
 \end{equation}
 By taking into account $\eqref{e:bounddeltaciro}$ and
that $\hat{v}_j\leq \delta_{j}$
$\calL^1$-a.e. on $I_1$, we have by \eqref{e:bounddeltaciro bis}
 \begin{equation}\label{e:terzaeq}
 \frac{1}{\eps_{k_j}}
 \varphi(\eps_{k_j}f^2(\hat{v}_j))
 \geq(1-\theta(\delta))\varphi'(0^+)
 f^2(\hat{v}_j)\,,
 \end{equation}
 where $\theta(\delta)\to 0^+$ as $\delta\to 0^+$.
 In particular, by \eqref{e:primaeq}-\eqref{e:terzaeq}
 we conclude that
 \begin{equation}\label{e:third reduction}
\liminf_{j\to\infty}\calG_j(\tilde{u}_j,\tilde{v}_j)\geq
\liminf_{j\to\infty}\calG_j
(\tilde{u}_j,\hat{v}_j)\geq(1-\theta(\delta))
\liminf_{j\to\infty}\widehat{\callF}_{\eta_j}(\tilde{u}_j,\hat{v}_j,I_1)\,,
 \end{equation}
 where $\widehat{\callF}_{\eta_j}:L^1(I_1,\R^2)\to[0,\infty]$ is the functional
 defined for every $(u,v)\in  H^1(I_1,\R\times [0,1])$ by
 \begin{equation}\label{e:second reduction}
\widehat{\callF}_{\eta_j}(u,v) =
\int_{I_1}\left( \eta_j \varphi'(0^+)
f^2(v)|u'|^2 + \frac{\PotDann(1-v)}{4\eta_j} +\eta_j|v'|^2\right) \dx\,,
 \end{equation}
 and $\infty$ otherwise.

Up to extracting a subsequence not relabeled, we may use \eqref{e:conv vjtilde} 
to find points
$-\frac12<x_1<0<x_2<\frac12$ such that $\tilde{v}_j(x_i)\to 1$ (so that $\hat{v}_j(x_i)\to 1$),
$\tilde{u}_j(x_i)\to u_0(x_i)$, $i\in\{1,2\}$, and moreover that
\[
\liminf_{j\to\infty}\widehat{\callF}_{\eta_j}(\tilde{u}_j,\hat{v}_j,I_1)
=\lim_{j\to\infty}\widehat{\callF}_{\eta_j}(\tilde{u}_j,\hat{v}_j,I_1)\,.
\]
In particular, with fixed any $\eta>0$, for all
$j$ sufficiently large by Cauchy-Schwartz inequality
we get for $g_\eta$ defined in \eqref{e:legieta}
\begin{align}\label{e:stima hatFj}
\widehat{\callF}_{\eta_j}(\tilde{u}_j,\hat{v}_j,I_1)&\geq
 \int_{x_1}^{x_2}\Big(\PotDann(1-\hat{v}_j)\big(
\varphi'(0^+)f^2(\hat{v}_j)|\tilde{u}_j'|^2+|\hat{v}_j|^2\big)\Big)^{\sfrac12}\dx\notag\\
&\geq g_\eta(|\tilde{u}_j(x_2)-\tilde{u}_j(x_1)|)\,.
\end{align}
In deriving the last inequality, we have used that the functional to be minimized in the definition of $g_\eta$ is invariant under reparametrization.

The conclusion in \eqref{e:blowup jump x}
then follows by \eqref{e:first reduction}, 
\eqref{e:third reduction}, \eqref{e:stima hatFj} and item (vi) in Proposition~\ref{p:lepropdig}.
\end{proof}
We gather the lower estimates on the diffuse and surface parts obtained in Propositions~\ref{p:diffuseliminf sigmabarra finito} and \ref{p:lowerboundBVsfc}, respectively, via a standard measure theoretic argument.
\begin{proposition}[Lower Bound inequality]\label{p:lbcompBV}
Under the assumptions of Theorem~\ref{t:finale} with $\FailureS\in(0,\infty)$,
for every $(u_{\eps},v_{\eps}),\,(u,v)\in L^1(\Omega,\R^2)$ with $(u_{\eps},v_{\eps})\to (u,v)$ in $L^1(\Omega,\R^2)$ we have
\begin{equation}\label{e:lbBV}
F_{\,\FailureS}(u,v) \le \liminf_{\eps\to0} \Functeps(u_{\eps}, v_{\eps}),
\end{equation}
where $\Functeps$ and $F_{\,\FailureS}$ are defined in \eqref{functeps} and \eqref{F0},
respectively.
\end{proposition}
\begin{proof}
Without loss of generality, we assume that the inferior limit in \eqref{e:lbBV} to be finite.
Thus, $(u_{\eps},v_\eps)\in H^1(\Omega,\R\times[0,1])$ for every $\eps>0$
sufficiently small, and moreover $u\in GBV(\Omega)$ and $v=1$ $\calL^1$-a.e. on $\Omega$
in view of Theorem~\ref{t:cptness sigma finito}.
It is sufficient to establish \eqref{e:lbBV} if $u\in BV(\Omega)$ by employing a standard trucation argument and the fact that the functionals $\Functeps$ are decreasing by truncation.

Thus, for $u\in BV(\Omega)$, we may consider the superadditive set function defined on $\calA(\Omega)$ by
\begin{equation*}
\mu(A):=\displaystyle\liminf_{\eps\to0}\Functeps(u_\eps,v_\eps,A),
\end{equation*}
if $A\in\calA(\Omega)$, and the Radon measure
\begin{equation*}
\nu:=\calL^1\res\Omega+\calH^{0}\res J_u+|D^cu|\,.
\end{equation*}
In particular, $\nu$ is the sum of three Radon measures concentrated
on mutually disjoint Borel sets $U_1,\,U_2,\,U_3$ partitioning $\Omega$.
Then, we may define two Borel functions $\psi_1,\psi_2:\Omega \to [0,\infty]$ by
\begin{equation*}
\psi_1(x)=
\begin{cases}
g(|[u](x)|)  & \textup{on $U_2$} \\
0  & $\textup{otherwise}$
\end{cases}
\qquad \textup{and}\qquad
\psi_2(x)=
\begin{cases}
\hsigmabarra^{**}(|u'(x)|)  & \textup{on $U_1$} \\
(\varphi'(0^+))^{\sfrac{1}{2}}\FailureS  & $\textup{on $U_3$}$ \\
0  & $\textup{otherwise.}$
\end{cases}
\end{equation*}
Propositions~\ref{p:diffuseliminf sigmabarra finito} and \ref{p:lowerboundBVsfc}
then imply for $i=1,2$ and  $A\in\calA(\Omega)$
\[
\mu(A)\geq\int_A\psi_i \dd\nu\,,
\]
thus \cite[Proposition~1.16]{Braides1998} yields that
\[
\mu(\Omega)\geq\int_\Omega(\psi_1\vee\psi_2)\text{d}\nu
=F_{\,\FailureS}(u,1)\,.
\qedhere
\]
\end{proof}

To show the upper bound inequality we follow in part the strategy in \cite[Proposition~5.2]{ContiFocardiIurlano2016,ContiFocardiIurlano2022}
and take advantage of the one-dimensional setting.
More generally, we establish it for the perturbed family of functionals
$\Functeps^{\kappa}:L^1(\Omega,\R^2)\times\calA(\Omega)\to[0,\infty]$
\begin{equation}\label{e:F eps kappa}
\Functeps^{\kappa}(u,v,A):=\Functeps(u,v,A)+\kappa_\eps\int_A|u'|^2\dx\,,
\end{equation}
if $(u,v)\in H^1(\Omega,\R\times[0,1])$ and $\infty$ otherwise,
where $\kappa_\eps=o(\eps)$ as $\eps\to 0$ and $\kappa_\eps\geq 0$, in order to gain coercivity for applications to Dirichlet boundary value problems (see Section~\ref{ss:Dirichlet}). We denote by ${\displaystyle{F'':=\Gamma(L^1)\text{-}\limsup_{\eps\to0}\Functeps^{\kappa}}}$.

We divide the argument into several steps, by providing first a rough bound for the diffuse part in the case of Sobolev functions and then obtaining the sharp bound optimizing upon the former rough one through a relaxation argument. The extension of the upper bound to piecewise Sobolev functions is done thanks to an explicit construction matching the surface energy density at jump points, finally the sharp bound for any $BV$ function is obtained again through relaxation.

In what follows it is convenient to consider the functional $H:L^1(\Omega)\times\calA(\Omega)\to[0,\infty]$ defined by
\begin{equation}\label{e:H bis}
    H(u,A):=\int_A\hsigmabarra(|u'|)\dx\,,
\end{equation}
if $u\in W^{1,1}(\Omega)$, and $\infty$ otherwise.
Given the continuity of $\hsigmabarra$ (cf. item (i) in Lemma~\ref{p:proprieta hsigma}) and the growth conditions in \eqref{minorconlin}, the functional $H$ in \eqref{e:H bis} is continuous with respect to the strong convergence in $W^{1,1}(\Omega)$. The same remark applies to the functional $\Functu_{\FailureS}$. We will take advantage of this fact by proving in several instances the upper bound inequality (cf. \eqref{e:limsuppone} below) on classes of functions which are dense in $L^1(\Omega)$ or in a stronger topology, and along which $H$ and $\Functu_{\FailureS}$ are continuous, respectively. The $L^1$ lower semicontinuity of $F''$ will allow to extend the validity of \eqref{e:limsuppone} to functions in the $L^1$ closure of such dense classes.
\begin{proposition}[Upper Bound inequality]\label{p:uppertilde sigma finito}
Under the assumptions of Theorem~\ref{t:finale} with $\FailureS\in(0,\infty)$,
for every $(u,v)\in L^1(\Omega;\R^2)$
\begin{equation}\label{e:limsuppone}
F''(u,v) \leq F_{\,\FailureS}(u,v)\,.
\end{equation}
\end{proposition}
\begin{proof}
It is clearly sufficent to assume $v=1$ $\calL^1$-a.e. on $\Omega$. We split the proof in different steps.

\noindent{\bf Step 1.} If $u(x)=\xi x+\eta$ is affine on $\Omega$ then for every interval $I\subseteq\Omega$
\begin{equation}\label{e:rough upper bound affine}
   F''(u,1,I)\leq
   H(u,I)=\hsigmabarra(|\xi|)\calL^1(I)\,.
\end{equation}
Let $\tau_\xi\in[0,\infty)$ be such that $\hsigmabarra(|\xi|)=\varphi(\sfrac1{\tau_\xi})\xi^2 +\frac{\FailureS^2}{4}\tau_\xi$
(with the convention that $\varphi(\sfrac1\tau)$ is extended by continuity as $\varphi(\infty)\in(0,\infty)$ in $\tau=0$).
Then set $u_\eps=u$ and $v_\eps=\lambda_\eps$ where
$\eps f^2
(\lambda_\eps)=\eps\frac{\hatf(\lambda_\eps)}{Q(1-\lambda_\eps)}=\sfrac1{\tau_\xi}$
(note that $\lambda_\eps=1$ if $\tau_\xi=0$). By the very definition
$\lambda_\eps\to 1$ as $\eps\to 0$. Therefore
\begin{align*}
\Functeps^\kappa(u_\eps,v_\eps,I)&=\Big(\varphi(\sfrac1{\tau_\xi})\xi^2+\kappa_\eps\xi^2+\frac{\PotDann(1-\lambda_\eps)}{4\eps}\Big)\calL^1(I)\\
&= \Big(\varphi(\sfrac1{\tau_\xi})\xi^2+\kappa_\eps\xi^2
+\frac{\PotDann(1-\lambda_\eps)}{4 Q(1-\lambda_\eps)}\hatf(\lambda_\eps)\tau_\xi\Big)\calL^1(I)\,,
\end{align*}
the conclusion then follows by taking into account \eqref{e:FailureS def},
and the fact that $\hatf(\lambda_\eps)\to \hatf(1)=1$.
\medskip

\noindent{\bf Step 2.} The inequality in \eqref{e:rough upper bound affine} holds if $u$ is piecewise affine on $\Omega=(a,b)$, that is $u\in C^0(\Omega)$ and there are $\{a_i\}_{i=1}^N$ where $a_1=a$, $a_N=b$ and $a_i<a_{i+1}$ for $i\in\{1,\ldots N-1\}$ such that
\[
u(x)=\sum_{i=1}^N(\xi_ix+\eta_i)\chi_{[a_i,a_{i+1})}(x)\,.
\]
First, we show explicitly the claim for the (dense) subclass of piecewise affine functions
$u$ such that $u'=0$ on $(a_i,a_{i+1})$ for $i\in\{2,\ldots,N-1\}$ even.
Let $\delta\in(0,\min_i(a_{i+1}-a_i))$, set $\Omega_{i,\delta}:=(a_i-\sfrac\delta2,a_{i+1}+\sfrac\delta2)$, and let $\{\phi_i\}_{i=1}^N$ be a partition of unity subordinated to the covering $\{\Omega_{i,\delta}\}_{i=1}^N$ of $\Omega$, i.e. $\phi_i\in C_c^\infty(\Omega_{i,\delta},[0,1])$, $\phi_i|_{(a_i+\sfrac\delta2,a_{i+1}-\sfrac\delta2)}=1$,
$\max_{1\leq i\leq N}\|\phi_i'\|_{C^0(\R)}\leq\sfrac C\delta$, for some $C>0$, and $\sum_{i=1}^N\phi_i(x)=1$ for every $x\in\Omega$.

Set $u_\eps:=u$ and $v_\eps:=\sum_{i=1}^N\lambda_{\eps,i}\phi_i$, where
$\eps f^2(\lambda_{\eps,i})=\sfrac1{\tau_{\xi,i}}$ for every $i\in\{1,\ldots,N\}$
with $\lambda_{\eps,i}\in[0,1]$ (using the notation introduced in the previous step).
In particular, $\lambda_{\eps,i}=1$ for $i\in\{2,\ldots,N-1\}$ even because
$\xi_i=0$ and $\hsigmabarra(0)=0$. Therefore, from Step~1 we get for
$\delta$ sufficiently small
\begin{align}\label{e:stima piecewise affine}
\Functeps^\kappa(u_\eps,v_\eps,I)&\leq H(u,I)+o_\eps
+(\delta\varphi(\infty)+\kappa_\eps\calL^1(\Omega))\sum_{i=1}^N|\xi_i|^2\nonumber\\&
+CN\frac\eps\delta
+C\sum_{i=2}^{N-1}\int_{a_i-\sfrac\delta2}^{a_i+\sfrac\delta2}
\frac{Q(1-v_\eps)}\eps \dx\,,
\end{align}
where $C>0$ is a universal constant, $o_\eps\to 0$ as $\eps\to0$,
and we have used \eqref{e:FailureS def} and
the convergence $v_\eps\to1$ in $L^\infty(\Omega)$.
Next, note that $\phi_{i-1}(x)+\phi_i(x)=1$ on $(a_i-\sfrac\delta2,a_i+\sfrac\delta2)$ for $\delta$ sufficiently small, so that on such a set $1-v_\eps=(1-\lambda_{\eps,i-1})\phi_{i-1}+(1-\lambda_{\eps,i})\phi_i$.
In addition, using that $u'=0$ on $(a_i,a_{i+1})$ for even $i\in\{2,\ldots,N-1\}$, then $1-v_\eps=(1-\lambda_{\eps,i})\phi_i$ if $i\geq3$ is odd, and
$1-v_\eps=(1-\lambda_{\eps,i-1})\phi_{i-1}$ if $i\geq2$ is even.
Hence, if $i$ is odd on $(a_i-\sfrac\delta2,a_i+\sfrac\delta2)$ we have
\[
Q(1-v_\eps)=Q((1-\lambda_{\eps,i})\phi_i)\leq Q(1-\lambda_{\eps,i})=\eps\tau_{\xi_i}
\hatf(\lambda_{\eps,i})\,,
\]
where we have used that $Q$ is non-decreasing in a neighbourhood of the origin (cf.~(Hp~1)),
and  $\phi_i\in[0,1]$. Clearly, an analogous statement holds if $i$ is even.
Thanks to this estimate and to \eqref{e:stima piecewise affine} we conclude that
\begin{align*}
\Functeps^\kappa(u_\eps,v_\eps,I)&\leq H(u,I)+o_\eps
+(\delta\varphi(\infty)+\kappa_\eps\calL^1(\Omega))\sum_{i=1}^N|\xi_i|^2\nonumber\\&
+CN\frac{\eps}{\delta}+C\delta\sum_{i=1}^N\tau_{\xi_i}
\hatf(\lambda_{\eps,i})\,,
\end{align*}
from which we conclude \eqref{e:rough upper bound affine} by taking first the superior limit as $\eps\to0$ and then as $\delta\to0$.

Finally, given any piecewise affine function $u$ with an underlying partition
$\{a_1,..,a_N\}$, consider
\begin{equation*}
    u_j(x)=
    \begin{cases}
        u(x)\;&\textup{if $x \in [a_1,a_2]$} \\
        u(a_i)\;&\textup{if $x \in [a_i,a_i+\sfrac1j]$} \\
        u(a_{i+1})+\frac{u(a_{i+1})-u(a_i)}{a_{i+1}-a_i-\sfrac1j}(x-a_{i+1}) \;&\textup{if $x\in [a_i+\sfrac1j,a_{i+1}]$} \\
    \end{cases}
\end{equation*}
for $i\in \{2,..,N-1\}$, where $0<\sfrac1j<\min_i(a_{i+1}-a_i)$.
It is easy to show that $u_j\to u$ in $L^\infty(\Omega)$ and $u'_j \to u'$ in $L^p(\Omega)$ for every $p\in [1,\infty)$ as $j\to\infty$. We conclude using the density argument explained at
the beginning of the proof.
\medskip

\noindent{\bf Step~3.} The inequality in  \eqref{e:limsuppone} holds on $W^{1,1}(\Omega)$.

\noindent We use first the continuity of $H$ with respect to the strong $W^{1,1}(\Omega)$ convergence together with the density of piecewise affine functions in $W^{1,1}(\Omega)$ (this is easily established for the dense class of smooth functions, see also \cite[Proposition~2.1 in Chapter~X]{ekeland1999convex}) to extend the validity of inequality \eqref{e:rough upper bound affine} to every $u\in W^{1,1}(\Omega)$ via the density argument explained at the beginning of the proof.

We obtain the bound $F''(u,1)\leq F_{\,\FailureS}(u,1)$ for $u\in W^{1,1}(\Omega)$ using a classical relaxation result (cf. for instance \cite[Theorem~4.4.1 and Remark~4.4.5]{Buttazzo89},
\cite[Statement III.7]{AcerbiFusco1984semicontinuity}).
\medskip

\noindent{\bf Step~4.} The inequality in \eqref{e:limsuppone} holds if $u$ is $SBV(\Omega)$ with $\calH^0(J_u)<\infty$.

We explicit the construction first for
$u\in SBV(\Omega)$ with $J_u=\{x_0\}$ such that
$u=u(x^-_0)$ on $(x_0-2\lambda,x_0)$ and $u=u(x^+_0)$ on $[x_0,x_0+2\lambda)$ for some $\lambda>0$ with
$(x_0-2\lambda,x_0+2\lambda)\subseteq \Omega$, where $u(x_0^-),u(x_0^+)$ are respectively the left and the right limit of $u$ in $x_0$ (without loss of generality we  suppose that $u(x_0^-)<u(x_0^+)$).

A simple contradiction argument yields that it is sufficient to show that for every infinitesimal sequence
$\{\eps_j\}_{j\in\N}$ there are a subsequence
$\{\eps_{j_k}\}_{k\in\N}$ and $(u_{\eps_{j_k}},v_{\eps_{j_k}})\to (u,1)$ in $L^1(\Omega,\R^2)$ for which
\[
\limsup_{k\to\infty}\mathcal{F}_{\eps_{j_k}}^\kappa(u_{\eps_{j_k}},v_{\eps_{j_k}})\leq F_{\FailureS}(u,1)\,.
\]
Fix, $\{\eps_j\}_{j\in\N}$ infinitesimal, and consider $I_1:=(a,x_0)$ and $I_2:=(x_0,b)$.
Being $u\in W^{1,1}(I_1\cup I_2)$, there are sequences $(u_{\eps_j}^{(i)},v_{\eps_j}^{(i)})\to (u,1)$ in $L^1(I_i,\R^2)$
such that $F''(u,I_i)=\lim_j\Functu_{\eps_j}^\kappa
(u_{\eps_j}^{(i)},v_{\eps_j}^{(i)},I_i)$.
We may then extract a subsequence such that
$(u_{\eps_{j_k}}^{(i)},v_{\eps_{j_k}}^{(i)})\to (u,1)$ $\calL^1$-a.e. on $I_i$ for $i\in\{1,2\}$, as well.
For the sake of notational simplicity in what follows we denote 
$(u_{\eps_{j_k}}^{(i)},v_{\eps_{j_k}}^{(i)})$ simply by
$(u_k^{(i)},v_k^{(i)})$. By a.e. convergence we can find points $x_1\in I_1\cap(x_0-2\lambda,x_0-\lambda)$ and
$x_2\in I_2\cap(x_0+\lambda,x_0+2\lambda)$, such that $(u_k^{(i)}(x_i),v_k^{(i)}(x_i))\to(u(x_i),1)$, with
$u(x_1)=u(x_0^-)$ and $u(x_2)=u(x_0^+)$.

By item (v) in Proposition~\ref{p:lepropdig} for every $\eta >0$ there exist
$T_{\eta}>0$ and $(\gamma_{\eta},\beta_{\eta})\in \mathfrak{U}_{|[u](x_0)|}(0,T_{\eta})$ such that
    \begin{equation}\label{e:gammaoverline sigma finito}
    \int_0^{T_{\eta}} \left(\varphi'(0^+)f^2(\beta_{\eta})|\gamma_{\eta}'|^2+ \frac{\PotDann(1-\beta_{\eta})}{4}+ |\beta_{\eta}'|^2
     \right)\dx\leq
      g(|[u](x_0)|)+\eta\,.
    \end{equation}
   Set $A^{\eta}_k:=\left (x_0-\frac{\eps_{j_k} T_\eta}{2}, x_0+\frac{\eps_{j_k} T_\eta}{2}\right)$, then $A^{\eta}_k\subseteq (x_1+\lambda,x_2-\lambda)$ for $k$ sufficiently big. Then define
   $\{(u^{\eta}_k,v^{\eta}_k)\}_{k\in\N}$ by
   \begin{equation*}
       u^{\eta}_k(x)=
    \begin{cases}
    u_k^{(1)} & x\in (a,x_1)\\
    u_k^{(1)}(x_1)+(u(x_0^-)-u_k^{(1)}(x_1))\lambda^{-1}(x-x_1)
    & x\in[x_1,x_1+\lambda] \\
    u(x_0^-) & x\in (x_1+\lambda,x_0-\frac{\eps_{j_k}}2 T_\eta]\\
    u(x_0^-)+\gamma_{\eta}\left(\frac{x-x_0}{\eps_{j_k}}+\frac{T_{\eta}}{2}\right) & x\in A^{\eta}_k \\
    u(x_0^+) & x\in [x_0+\frac{\eps_{j_k}}2 T_\eta,x_2-\lambda),\\
    u_k^{(2)}(x_2)+(u_k^{(2)}(x_2)-u(x_0^+))\lambda^{-1}(x-x_2)
    & x\in[x_2-\lambda,x_2] \\
    u_k^{(2)} & x\in (x_2,b)\,,
    \end{cases}
   \end{equation*}
   and
   \begin{equation*}
       v^{\eta}_k(x)=
    \begin{cases}
        v_k^{(1)} & x\in (a,x_1)\\
        \zeta_k^{(1)} & x\in[x_1,x_0-\frac{\eps_{j_k} T_\eta}{2}]\\
        \beta_{\eta}\left(\frac{x-x_0}{\eps_{j_k}}+\frac{T_{\eta}}{2}\right) & x\in A^{\eta}_k \\
        \zeta_k^{(2)} & x\in[x_0+\frac{\eps_{j_k} T_\eta}{2},x_2] \\
        v_k^{(2)} & x\in (x_2,b)\,,
    \end{cases}
   \end{equation*}
where $\zeta_k^{(1)}(x)=(v_k^{(1)}(x_1)+\frac1{\eps_{j_k}}(x-x_1))\wedge 1$ and $\zeta_k^{(2)}(x)=(v_k^{(2)}(x_2)+\frac1{\eps_{j_k}}(x_2-x))\wedge 1$. Therefore, if  $\zeta_k^{(i)}(x)<1$ then $x\in[x_1,x_1+\eps_{j_k}(1-v_k^{(1)}(x_1))]$ if $i=1$, and $x\in[x_2-\eps_{j_k}(1-v_{\eps_{j_k}}^{(2)}(x_2)),x_2]$ if $i=2$.
 Note that $(u^{\eta}_k,v^{\eta}_k)\in H^1(\Omega,\R\times [0,1])$ thanks to  the assumptions on $u$. Moreover, $(u^{\eta}_k,v^{\eta}_k)\to (u,1)$ in $L^1(\Omega,\R^2)$ as $\gamma_{\eta}$, $\beta_{\eta} \in L^{\infty}(0,T_{\eta})$.
 To estimate the energy of $\{(u^{\eta}_k,v^{\eta}_k)\}_{k\in\N}$ we start with the contribution on $I_1\cup I_2\setminus A^{\eta}_k$:
 \begin{align}\label{e:stima energia 0 sigma finito}
\limsup_{k\to\infty}\mathcal{F}_{\eps_{j_k}}^\kappa&(u^{\eta}_k,v^{\eta}_k,
(a,x_1)\cup(x_2,b))\notag\\&=
\limsup_{k\to\infty}\big(\mathcal{F}_{\eps_{j_k}}^\kappa(u^{(1)}_k,v^{(1)}_k,
(a,x_1))+
\mathcal{F}_{\eps_{j_k}}^\kappa(u^{(2)}_k,v^{(2)}_k,
(x_2,b)\big)\nonumber\\
&\leq F''(u,1,I_1)+F''(u,1,I_2)\leq \int_\Omega \hsigmabarra^{**}(|u'|)\dx\,,
 \end{align}
 by the previous step as the intervals are disjoint.
 Next, we use a change of variable to estimate the contribution on $A^{\eta}_k$ as follows
\begin{align*}
\mathcal{F}_{\eps_{j_k}}^\kappa&(u^{\eta}_k,v^{\eta}_k,A^{\eta}_k)
\\&
\leq\int_0^{T_{\eta}}\left(\frac1{\eps_{j_k}} (\varphi(\eps_{j_k}f^2(\beta_{\eta}))+\kappa_{\eps_{j_k}})
|\gamma_{\eta}'|^2+ \frac{\PotDann(1-\beta_{\eta})}{4}
+ |\beta_{\eta}'|^2\right)\dx\,.
\end{align*}
Being $\varphi$ (right) differentiable in $0$ and bounded, there is $C>0$ such that
$\varphi(t)\leq Ct$ for all $t\geq 0$. Hence by Lebesgue dominated convergence theorem, by $\kappa_{\eps_{j_k}}=o(\eps_{j_k})$ as $k\to \infty$, and by \eqref{e:gammaoverline sigma finito} we conclude that
\begin{equation}\label{e:stima energia 1 sigma finito}
\limsup_{k\to\infty}
\mathcal{F}_{\eps_{j_k}}^\kappa(u^{\eta}_k,v^{\eta}_k,A^{\eta}_k)
\leq g(|[u](x_0)|)+\eta\,.
\end{equation}
Finally, we are left with estimating the energy on the set
$\Omega^{\eta}_k:=((a,x_1)\cup(x_2,b))\setminus A^{\eta}_k$.
By construction both $u^{\eta}_k$ and $v^{\eta}_k$ are piecewise affine on such a set,
a direct computation then yields
 \begin{align*}
\mathcal{F}_{\eps_{j_k}}^\kappa&(u^{\eta}_k,v^{\eta}_k,\Omega^{\eta}_k)\leq
\frac1{\lambda}(\varphi(\infty)+\kappa_{\eps_{j_k}})\big(
(u(x_0^+)-u_k^{(1)}(x_1))^2+(u(x_0^+)-u_k^{(2)}(x_2))^2\big)
\nonumber\\
&+\frac{\kappa_{\eps_{j_k}}}{\eps_{j_k}}\int_0^{T_\eta}|\gamma'_\eta|^2\dx
+(2-v^{(1)}_k(x_1)-v^{(2)}_k(x_2))(\sup_{[0,1]}\PotDann+1)\,.
 \end{align*}
Therefore, recalling that $u$ is constant on $(x_0-2\lambda,x_0)$ and on $(x_0,x_0+2\lambda)$,
by the choice of $x_1$ and $x_2$, and as $\kappa_\eps=o(\eps)$ as $\eps\to0$ we conclude that
 \begin{equation}\label{e:stima energia 2 sigma finito}
\limsup_k\mathcal{F}_{\eps_{j_k}}^\kappa(u^{\eta}_k,v^{\eta}_k,\Omega^{\eta}_k)=0
\end{equation}
By collecting  \eqref{e:stima energia 0 sigma finito}- \eqref{e:stima energia 2 sigma finito} and letting $\eta\to0$, \eqref{e:limsuppone} follows at once.

To remove the assumption that $u$ is piecewise constant close to the jump point $x_0$, we define the sequence
\begin{equation}\label{e:u costante vicino salto}
u_j(x)=u(x)\chi_{\Omega\setminus[x_0-\frac1j,x_0+\frac1j]}
+u(x_0-\textstyle{\frac1j})\chi_{(x_0-\frac1j,x_0]}+u(x_0+\textstyle\frac1j)\chi_{(x_0,x_0+\frac1j)}\,.
\end{equation}
We have $u_j\to u$ in $L^1(\Omega)$ and $u_j'=u'\chi_{\Omega\setminus[x_0-\frac1j,x_0+\frac1j]}$ $\calL^1$-a.e. on $\Omega$, and thus we conclude using the density argument explained before the statement.

Finally, since our construction modifies recovery sequences of $F''$ on each sub-interval on which $u$ is in $W^{1,1}$ only in a neighbourhood of the endpoints, we can extend the validity of \eqref{e:limsuppone} to every $u\in SBV(\Omega)$ with $\calH^0(J_u)<\infty$ arguing locally.
\medskip

\noindent{\bf Step~5.} The inequality in \eqref{e:limsuppone} holds if $u\in SBV(\Omega)$.

\noindent We claim that each function $u\in SBV(\Omega)$ can be approximated with a sequence of functions $u_j\in SBV(\Omega)$ and $\mathcal{H}^0(J_{u_j})<\infty$ for every $j\in\N$, converging to $u$ in $L^1(\Omega)$ and with energies $\Functu_{\FailureS}(u_j)\to\Functu_{\FailureS}(u)$. Indeed, consider the decomposition $u=u^{(a)}+u^{(s)}$, where $u_a\in W^{1,1}(\Omega)$ and $u_s(x)=\sum_{y\in J_u\cap(a,x]}[u](y)$ is piecewise constant, and define
$u_j:=u^{(a)}+\sum_{y\in I_j\cap(a,x]}[u](y)$, where $I_j:=\{y\in J_u:\,|[u](y)|\geq\sfrac1j\}$.
\medskip

\noindent{\bf Step~6.} The inequality in \eqref{e:limsuppone} holds if $u\in GBV(\Omega)$.

\noindent Note that if \eqref{e:limsuppone} holds on $BV(\Omega)$, we may conclude it on $GBV(\Omega)$ by means of the sequence of truncations defined in \eqref{e:truncations}, together with the standard density argument.

To show inequality \eqref{e:limsuppone} on $BV(\Omega)$ we use a relaxation argument.
Indeed, thanks to Step~5 we may apply \cite[Theorem~3.1]{bou-bra-but} to get that
the relaxed functional with respect to the weak$^*$ $BV$ topology of $\widetilde{\Functu}_{\FailureS}:BV(\Omega)\to[0,\infty]$ defined
as $\widetilde{\Functu}_{\FailureS}=\Functu_{\FailureS}$ on $SBV(\Omega)$ and $\infty$ otherwise is given by
\begin{equation}\label{e:stima relax bou bra but}
\widetilde{\Functu}_{\FailureS}(u)=
\int_\Omega (\hsigmabarra^{**}\triangledown g^0)(|u'|)\dx
+\int_{J_u}((\hsigmabarra^{**})^\infty\triangledown g)(|[u]|)\mathrm{d}\mathcal{H}^0
+\int_\Omega (\hsigmabarra^{**}\triangledown g^0)^\infty\mathrm{d}D^cu\,,
\end{equation}
for every $u\in BV(\Omega)$, where
$\psi_1\triangledown\psi_2(t):=\inf\{\psi_1(x)+\psi_2(t-x):\,x\in[0,t]\}$ is the usual infimal convolution of two functions, and for every $s\geq0$
\[
g^0(s):=\limsup_{t\to0}\frac{g(st)}{t},\quad
(\hsigmabarra^{**})^\infty(s):=
\lim_{t\to\infty}\frac{\hsigmabarra^{**}(st)}{t}\,.
\]
In particular, $g^0(s)=(\hsigmabarra^{**})^\infty(s)
=\FailureS (\varphi'(0^+))^{\sfrac12} s$ for every $s\geq 0$
by Lemma~\ref{p:proprieta hsigma} and \eqref{e:strenght}, respectively.
We have $\psi_1\triangledown\psi_2\leq \psi_1\wedge\psi_2$ by the very definition. On the one hand, item (ii) in Proposition~\ref{p:lepropdig} implies $g\leq(\hsigmabarra^{**})^\infty$ so that $(\hsigmabarra^{**})^\infty\triangledown g\leq g$ and actually the equality holds by sub-additivity of $g$. On the other hand, \eqref{minorconlin} in Lemma~\ref{p:proprieta hsigma} implies $\hsigmabarra^{**}\leq g^0$ so that $\hsigmabarra^{**}\triangledown g^0\leq\hsigmabarra^{**}$ and actually the equality holds thanks to the convexity of $\hsigmabarra^{**}$, and the fact that if $\zeta$ belongs to the subdifferential of $\hsigmabarra^{**}$ in $t$, then $\zeta\leq g^0(1)$ by $g^0(s)=(\hsigmabarra^{**})^\infty(s)
=\FailureS (\varphi'(0^+))^{\sfrac12} s$.

The conclusion then follows at once from \eqref{e:stima relax bou bra but} by taking into account that $F''\leq \widetilde\Functu_{\FailureS}$ on $SBV(\Omega)\times\{1\}$ by Step~5, and that $F''$ is $L^1$ lower semicontinuous.
\end{proof}

\subsection{Case \texorpdfstring{$\FailureS=\infty$}{...}}

In this section we consider the case $\FailureS=\infty$.

\begin{theorem}[Compactness]\label{t:cptness sigma infinito}
Under the assumptions of Theorem~\ref{t:finale} with $\FailureS=\infty$,
let $\{(u_{\eps},v_{\eps})\}_{\eps>0} \in L^1(\Omega,\R^2)$ be such that
\begin{equation}\label{e:supremumnorme sigma infinito}
    \sup_{\eps >0} (\Functeps(u_{\eps},v_{\eps})+\|u_{\eps}\|_{L^1(\Omega)})<\infty\,,
\end{equation}
then there are 
$\{\eps_k\}_{k\in \N}$ and $u\in L^1\cap GSBV(\Omega)$ with $u'\in L^2(\Omega)$ such that $(u_{\eps_k},v_{\eps_k})\to (u,1)$ $\calL^1$-a.e. on $\Omega$.
If, moreover, $(u_\eps)_\eps$ is equi-integrable then $u_{\eps_k}\to u$ in $L^1(\Omega)$.
\end{theorem}
\begin{proof}
The proof follows the same strategy of Theorem~\ref{t:cptness sigma finito}.
We adopt the same notation used there and we highlight only the main changes.
The convergences of $(v_\eps)_\eps$ to $1$ in measure on $\Omega$, and of a subsequence $(u_{\eps_k})_k$ to some $u\in L^1\cap GBV(\Omega)$ follow from the same argument. The $\calL^1$-a.e. convergence on $\Omega$, and the $L^1(\Omega)$ convergence if $(u_\eps)_\eps$ is equi-integrable, are analogous.

We need only to show that $u\in GSBV(\Omega)$ and that
$u'\in L^2(\Omega)$. With this aim we note that
in case $\FailureS=\infty$, the estimate
$H_\delta(\tilde u_{\eps,\delta})\leq \Functeps(u_\eps,v_\eps)$ established in Lemma~\ref{l:suppenergy} holds with bulk density $h_{\sfrac1\gamma-1}$, for every $\gamma\in(0,1)$ (cf. \eqref{e:Haccagrand sigma infinito}).
Thus, in view of Proposition~\ref{p:CFI_sci} we conclude that for every $M>0$, $\gamma\in(0,1)$ and $\delta\in(0,\delta_\gamma)$
\begin{equation*}
\zeta_1(\delta)\int_\Omega h^{**}_{\sfrac1\gamma-1}(|(u^M)'|)\dx + \zeta_1(\delta)\left(\frac1\gamma-1\right) |D^cu^M|(\Omega) \leq C\,.
\end{equation*}
By letting first $\delta\to 0$ and then $\gamma\to0$ we conclude $|D^cu^M|(\Omega)=0$
for every $M\in\N$, that is equivalently
$u^M\in SBV(\Omega)$, from which we conclude that
$u\in GSBV(\Omega)$.
In addition, the latter estimate and Lemma~\ref{p:proprieta hsigma} imply
\[
\varphi(\infty)\int_\Omega |(u^M)'|^2\dx\leq C\,.
\]
The conclusion then follows at once by letting $M\to\infty$.
\end{proof}

The lower bound inequality in the case $\FailureS=\infty$
follows as the analogous result in Proposition~\ref{p:diffuseliminf sigmabarra finito}.
\begin{proposition}\label{p:diffuseliminf sigmabarra infinito}
Under the assumptions of Theorem~\ref{t:finale} with $\FailureS=\infty$,
for every $(u_{\eps},v_{\eps}),\,(u,v)\in L^1(\Omega,\R^2)$ with $(u_{\eps},v_{\eps})\to (u,v)$ in $L^1(\Omega,\R^2)$, we have
\begin{equation}\label{e:liminf diffuse sigmabarra infinito}
    \int_A h_\infty(|u'|)\dx \leq \liminf_{\eps \to 0} \Functeps(u_{\eps},v_{\eps},A)\,,
\end{equation}
for every $A\in\calA(\Omega)$.
\end{proposition}
\begin{proof}
First, let us assume the inferior limit on the right hand side of \eqref{e:liminf diffuse sigmabarra infinito} is finite, otherwise the claim is obvious.
Therefore, $(u_{\eps},v_{\eps})\in H^1(\Omega,\R\times[0,1])$ for $\eps>0$ sufficiently small, $u\in GSBV(\Omega)$ and $v=1$ $\calL^1$-a.e. by
Theorem~\ref{t:cptness sigma infinito}.
We show the result in case $u\in SBV(\Omega)$, the general case follows as in
Proposition~\ref{p:diffuseliminf sigmabarra finito}.

Fix $\gamma\in(0,1)$ and $A\in\calA(\Omega)$, using the notation and the results in  Lemma~\ref{l:suppenergy} we find $\delta_\gamma\in(0,1)$ such that for every
$\delta\in(0,\delta_\gamma)$ there is $\tilde{u}_{\eps,\delta}:=(\widetilde{u_\eps})_\delta\in SBV(A)$ (for every $\eps$ small enough) we have
\[
H_\delta(\tilde u_{\eps,\delta},A)\leq \Functeps(u_\eps,v_\eps,A)\,,
\]
with density $h_{\sfrac1\gamma}$ (cf. \eqref{e:Haccagrand sigma infinito}).
Thus, in view of Proposition~\ref{p:CFI_sci} we conclude that
\begin{equation}
    \zeta_1(\delta)\int_A h^{**}_{\sfrac1\gamma-1}(|u'|)\dx
    \leq \liminf_{\eps \to 0} \Functeps(u_{\eps},v_{\eps},A)\,,
\end{equation}
recalling that $|D^cu|(\Omega)=0$.
By letting first $\delta\to 0$ and then $\gamma\to0$ we conclude 
\eqref{e:liminf diffuse sigmabarra infinito} in view of
\eqref{e:limittruncat}, \eqref{e:hinftyapprox} in Lemma~\ref{p:proprieta hsigma}, and Beppo Levi's theorem.
\end{proof}
The proof of the lower bound inequality for the surface part takes advantage of Proposition~\ref{p:lowerboundBVsfc} and of \cite[Step~2 in Corollary~4.8]{ContiFocardiIurlanopgrowth}.
\begin{proposition}\label{p:lowerboundBVsfc sigmabarra infinito}
Under the assumptions of Theorem~\ref{t:finale} with $\FailureS=\infty$,
for every $(u_{\eps},v_{\eps}),\,(u,v)\in L^1(\Omega,\R^2)$ with $(u_{\eps},v_{\eps})\to (u,v)$ in $L^1(\Omega,\R^2)$, we have
\begin{equation}\label{lbjump sigmabarra infinito}
\int_{J_u{\cap A}}g(|[u]|)\dd\calH^{0}  \le
\liminf_{\eps\to0}\Functeps(u_\eps,v_\eps,A)
\end{equation}
for every $A\in\calA(\Omega)$, where $g$ is defined in \eqref{e:lagiyo}.
\end{proposition}
\begin{proof} The proof is similar to that of Proposition~\ref{p:diffuseliminf sigmabarra finito},
therefore we highlight only the necessary changes adopting the notation introduced there.
We proceed up to formula \eqref{e:conv vjtilde}, noticing that Theorem~\ref{t:cptness sigma infinito} implies that $u\in GSBV(\Omega)$ with $u'\in L^2(\Omega)$. Next, we change the argument as the truncation procedure in \eqref{e:bounddeltaciro}-\eqref{e:bounddeltaciro bis} does not work in this setting. As an outcome of the blow up procedure we are given a sequence
$(\tilde{u}_j,\tilde{v}_j)\to(u_0,1)$ in $L^1(I_1)$, where $I_1=(-\sfrac12,\sfrac12)$ and
$u_0(x)=u(x_0^-)\chi_{(-\frac12,0)} +u(x_0^+)\chi_{[0,\frac12)}$, such that
$\liminf_j\calG_j(\tilde{u}_j,\tilde{v}_j)<\infty$, $\calG_j$ defined in \eqref{e:Gj}.

Let $\gamma>0$ and $\delta>0$ be such that $\frac{\varphi(t)}{t}\geq \varphi'(0^+)-\gamma$
for all $t\in(0,\delta)$. Because of the continuity of $f$ on $[0,1)$, the assumption that $f$ is non-decreasing in a left neighbourhood of $t=1$, and $f(t)\to\infty$ as $t\to1^-$, then
$\{t\in[0,1):\,\eps_jf^2(t)>\delta\}=(\delta_j,1)$ for $j$ large, where $\delta_j\to 1^-$
as $j\to\infty$.
Hence, $\{t\in I_1:\,\eps_jf^2(\tilde{v}_j)>\delta\}
=\{t\in I_1:\,\tilde{v}_j>\delta_j\}=\cup_i(a_j^i,b_j^i)$, being $\tilde{v}_j\in H^1(I_1)$.
Consider the function
$\hat{u}_j:=\tilde{u}_j\chi_{\{\tilde{v}_j\leq\delta_j\}}
+\sum_i\tilde{u}_j(a_j^i)\chi_{(a_j^i,b_j^i)}$, then $\hat{u}_j\in SBV(I_1)$ with
$D\hat{u}_j=\tilde{u}_j'\chi_{\{\tilde{v}_j\leq\delta_j\}}\calL^1\res I_1
+\sum_i(\tilde{u}_j(b_j^i)-\tilde{u}_j(a_j^i))\delta_{t=b_j^i}$.
Take its absolutely continuous part $w_j$ in the standard decomposition
of $BV$ functions, namely
\[
w_j(x):=\tilde{u}_j(-\sfrac12)+\int_{-\sfrac12}^x \hat{u}_j'(t)\dt\,.
\]
We claim that $w_j\to u_0$ in $L^1(I_1)$ and that
\begin{align}\label{e:Gj comparison}
\hat{\mathcal{F}}_j(w_j,\tilde{v}_j)&:=\int_\Omega
\left(\eta_j(\varphi'(0^+)-\gamma)f^2(\tilde{v}_j)|w_j'|^2+\frac{\PotDann(1-v_j)}{4\eta_j}+\eta_j|v_j'|^2\right)\dx\notag\\
&\leq\calG_j(w_j,\tilde{v}_j)\leq \calG_j(\tilde{u}_j,\tilde{v}_j).
\end{align}
Indeed, first note that
\begin{align*}
\frac{\eta_j}{\eps_j}\varphi(\delta)
\int_{\{\tilde{v}_j>\delta_j\}}|\tilde{u}_j'|^2\dx\leq
\frac{\eta_j}{\eps_j}\int_{\{\tilde{v}_j>\delta_j\}}
\varphi(\eps_jf^2(\tilde{v}_j))|\tilde{u}_j'|^2\dx
\leq \calG_j(\tilde{u}_j,\tilde{v}_j)\,,
\end{align*}
and then that $w_j'=\hat{u}_j'\chi_{\{\tilde{v}_j\leq\delta_j\}}=\tilde{u}_j'
\chi_{\{\tilde{v}_j\leq\delta_j\}}$ $\calL^1$-a.e. on $I_1$.
Hence, $w_j\to u_0$ in $L^1(I_1)$ as
\[
\|w_j-\tilde{u}_j\|_{L^\infty(I_1)}\leq
\int_{\{\tilde{v}_j>\delta_j\}}|\tilde{u}_j'|\dx\leq
\left(\frac{\eps_j}{\varphi(\delta)\eta_j}\calG_j(\tilde{u}_j,\tilde{v}_j)
\right)^{\sfrac12}\,,
\]
and furthermore we have
\begin{align*}
\hat{\mathcal{F}}_j(w_j,\tilde{v}_j) \leq &\calG_j(w_j,\tilde{v}_j)
=\calG_j(\tilde{u}_j,\tilde{v}_j,\{\tilde{v}_j\leq\delta_j\})\\
&+\int_{\{\tilde{v}_j>\delta_j\}}\left(\frac{\PotDann(1-\tilde{v_j})}{\eps_j}+\eps_j|\tilde{v_j}'|^2\right)\dx
\leq \calG_j(\tilde{u}_j,\tilde{v}_j)\,,
\end{align*}
and thus \eqref{e:Gj comparison} follows.

The final argument is similar to that employed in Proposition~\ref{p:lowerboundBVsfc}
using Proposition~\ref{p:lepropdig infty} rather than Proposition~\ref{p:lepropdig},
and finally letting $\gamma\to0^+$.
\end{proof}
The lower bound inequality follows arguing analogously as in Proposition~\ref{p:lbcompBV}.
\begin{proposition}[Lower Bound inequality]\label{p:lbcompBV sigma infinito}
Under the assumptions of Theorem~\ref{t:finale} with $\FailureS=\infty$,
for every $(u_{\eps},v_{\eps}),\,(u,v)\in L^1(\Omega,\R^2)$ with $(u_{\eps},v_{\eps})\to (u,v)$ in $L^1(\Omega,\R^2)$ we have
\begin{equation}\label{e:lbBV sigma infinito}
F_\infty(u,v) \le \liminf_{\eps\to0} \Functeps(u_{\eps}, v_{\eps}),
\end{equation}
where $\Functeps$ and $F_\infty$ are defined in \eqref{functeps} and \eqref{F0},
respectively.
\end{proposition}

To conclude we show the upper bound inequality in this setting for the perturbed family $\{\Functeps^\kappa\}_{\eps>0}$ defined in \eqref{e:F eps kappa}.
We recall the notation $F''={\displaystyle{\Gamma(L^1)\text{-}\limsup_{\eps\to0}\Functeps^\kappa}}$.
\begin{proposition}[Upper Bound inequality]\label{p:uppertilde sigma infinito}
Under the assumptions of Theorem~\ref{t:finale} with $\FailureS=\infty$,
let $(u,v)\in L^1(\Omega;\R^2)$ then
\begin{equation}\label{e:limsuppone sigma infinito}
  F''(u,v) \leq F_\infty(u,v)\,.
\end{equation}
\end{proposition}
\begin{proof}
Without loss of generality we assume $u\in GSBV(\Omega)$ with $u'\in L^2(\Omega)$, and $v=1$ $\calL^1$-a.e. on $\Omega$, the inequality being trivial otherwise.
Moreover, we can reduce to $u\in SBV(\Omega)$ with $u'\in L^2(\Omega)$ by the density argument explained before Proposition~\ref{p:uppertilde sigma finito} by using the sequence of truncations in \eqref{e:truncations}.
Furthermore, we may even assume that $\calH^0(J_u)<\infty$ by using the construction in Step~5 of Proposition~\ref{p:uppertilde sigma finito}.

First, consider $u\in SBV(\Omega)$ with $u'\in L^2(\Omega)$, $J_u=\{x_0\}$ and
$u=u(x^-_0)$ on $(x_0-\lambda,x_0)$ and $u=u(x^+_0)$ on $[x_0,x_0+\lambda)$
for some $\lambda>0$ with $(x_0-\lambda,x_0+\lambda)\subseteq \Omega$, where $u(x_0^-),u(x_0^+)$ are respectively the left and the right limit of $u$ in $x_0$ (without loss of generality we can suppose that $u(x_0^-)<u(x_0^+)$), and argue as in Step~4 of Proposition~\ref{p:uppertilde sigma finito}.
In particular, the assumption that $u$ is constant near the jump point $x_0$ is not restrictive up to a density argument and the construction in \eqref{e:u costante vicino salto}.

By item (v) in Proposition~\ref{p:lepropdig infty} for every $\eta >0$ there exist
$T_{\eta}>0$ and $(\gamma_{\eta},\beta_{\eta})\in \mathfrak{U}_{|[u](x_0)|}(0,T_{\eta})$ such that
    \begin{equation}\label{e:gammaoverline}
    \int_0^{T_{\eta}} \left(\varphi'(0^+)f^2(\beta_{\eta})|\gamma_{\eta}'|^2+ \frac{\PotDann(1-\beta_{\eta})}{4}+ |\beta_{\eta}'|^2
     \right)\dx\leq
      g(|[u](x_0)|)+\eta\,.
    \end{equation}
   Set $A^{\eta}_\eps:=\left (x_0-\frac{\eps T_\eta}{2}, x_0+\frac{\eps T_\eta}{2}\right)$, then $A^{\eta}_{\eps}\subseteq (x_0-\lambda,x_0+\lambda)$ for $\eps$ sufficiently small. Define $\{(u^{\eta}_\eps,v^{\eta}_\eps)\}_{\eps}$ by
   \begin{equation*}
       u^{\eta}_\eps(x)=
    \begin{cases}
        u(x_0^-)+\gamma_{\eta}\left(\frac{x-x_0}{\eps}+\frac{T_{\eta}}{2}\right) & x\in A^{\eta}_\eps \\
        u & x\in \Omega \setminus A^{\eta}_\eps
    \end{cases}
   \end{equation*}
   and
   \begin{equation*}
       v^{\eta}_\eps(x)=
    \begin{cases}
        \beta_{\eta}\left(\frac{x-x_0}{\eps}+\frac{T_{\eta}}{2}\right) &
        x\in A^{\eta}_\eps \\
        1 & x\in \Omega \setminus A^{\eta}_\eps\,.
    \end{cases}
   \end{equation*}
 Note that $(u^{\eta}_\eps,v^{\eta}_\eps)\in H^1(\Omega,\R\times [0,1])$ thanks to
 the assumptions on $u$. Moreover, $(u^{\eta}_\eps,v^{\eta}_\eps)\to (u,1)$ in $L^1(\Omega,\R^2)$ as $\gamma_{\eta}$, $\beta_{\eta} \in L^{\infty}(0,T_{\eta})$.

 Next, we estimate the energy of the family $\{(u^{\eta}_\eps,v^{\eta}_\eps)\}_{\eps}$.
 We start with the contribution on $\Omega \setminus \overline{A^{\eta}_\eps}$ by taking
 into account that $v^{\eta}_\eps=1$ and $u^{\eta}_\eps=u$ on such a set to get (recall that $\varphi(\eps f^2)$ is extended by continuity with value $\varphi(\infty)$ to $t=1$)
 \begin{equation}\label{e:stima energia 1}
\Functeps^\kappa(u^{\eta}_\eps,v^{\eta}_\eps,\Omega \setminus \overline{A^{\eta}_\eps})
\leq (\varphi(\infty)+\kappa_\eps)
\int_{\Omega \setminus \overline{A^{\eta}_\eps}}|u'|^2 \dx\,.
 \end{equation}
For the contribution on $A^{\eta}_\eps$ we change variable to get
\begin{align*}
\Functeps^\kappa(u^{\eta}_\eps,v^{\eta}_\eps,A^{\eta}_\eps)
\leq\int_0^{T_{\eta}}\left(\frac1{\eps} (\varphi(\eps f^2(\beta_{\eta}))+\kappa_\eps)|\gamma_{\eta}'|^2+ \frac{\PotDann(1-\beta_{\eta})}{4}
+ |\beta_{\eta}'|^2\right)\dx\,.
\end{align*}
Being $\varphi$ (right) differentiable in $0$ and bounded, there is $C>0$ such that
$\varphi(t)\leq Ct$ for all $t\geq 0$. Hence by Lebesgue dominated convergence theorem, $\kappa_\eps=o(\eps)$ as $\eps\to 0$, and by \eqref{e:gammaoverline} we conclude that
\begin{equation}\label{e:stima energia 2}
\limsup_{\eps\to0}
\Functeps^\kappa(u^{\eta}_\eps,v^{\eta}_\eps,A^{\eta}_\eps)
\leq g(|[u](x_0)|)+\eta\,,
\end{equation}
By collecting  \eqref{e:stima energia 1} and \eqref{e:stima energia 2} and letting $\eta\to0$, \eqref{e:limsuppone sigma infinito} follows at once.

Finally, arguing locally we can extend the validity of \eqref{e:limsuppone sigma infinito} to every $u\in SBV^2(\Omega)$ with $\calH^0(J_u)<\infty$ (cf. Step~4 of Proposition~\ref{p:uppertilde sigma finito}).
\end{proof}

\subsection{Dirichlet boundary values problem}\label{ss:Dirichlet}

In this section we impose Dirichlet boundary conditions on the approximating energies, determine the related $\Gamma$-limit and discuss the convergence of the related minimum problems.
With this aim define $\FunctepsDir:L^1(\Omega,\R^2)\times\calA(\Omega)\to[0,\infty]$ by
\begin{equation}\label{functepsDir}
 \FunctepsDir(u,v,A):= 
\Functeps^\kappa(u,v,A)
\end{equation}
if $(u,v)\in H^1(\Omega, \R\times[0,1])$ and
$u(a^+)=0$, $u(b^-)=L$, $v(a^+)=v(b^-)=1$,
and $\infty$ otherwise, where $\Omega=(a,b)$ and $\Functeps^\kappa$ is defined in \eqref{e:F eps kappa}.
Here, $\kappa_\eps=o(\eps)$ and it is strictly positive.
\begin{theorem}\label{t:Dirichlet}
Assume (Hp~$1$)-(Hp~$4$) hold with $\FailureS\in(0,\infty]$, and
let $\FunctepsDir$ be the functional defined in \eqref{functepsDir}.
Then, for all $(u,v)\in L^1(\Omega,\R^2)$
\begin{equation}\label{e:Gamma Dir}
\Gamma({L^1})\text{-}\lim_{\eps\to0}\FunctepsDir(u,v)=D_{\,\FailureS}(u,v)\,,
\end{equation}
where
\begin{equation}\label{D0}
D_{\,\FailureS}(u,v):=
\begin{cases}
\Functu_{\FailureS}(u)+ g(|u(a^+)|)+g(|u(b^-)-L|)
 & \textup{if $v=1$ $\mathcal{L}^1$-a.e. on $\Omega$} \cr
\infty & \textup{otherwise}\,.
\end{cases}
\end{equation}
Moreover, if $(u_\eps,v_\eps)\in\textup{argmin}_{L^1(\Omega,\R^2)}\FunctepsDir$, there are a subsequence (not relabeled) and a function $u\in GBV(\Omega)$ such that $(u_\eps,v_\eps)\to (u,1)$ in $L^1(\Omega,\R^2)$ and
\begin{equation}\label{e:conv min}
 \lim_{\eps\to0}  \FunctepsDir (u_\eps,v_\eps)=D_{\,\FailureS}(u,1)\,.
\end{equation}
\end{theorem}
\begin{proof}
We first show how to deduce \eqref{e:Gamma Dir} thanks to the results in
Theorem~\ref{t:finale}. With this aim, for every $(u,v)\in H^1(\Omega,\R^2)$
denote by $(U,V)$ the extensions given by
\[
U(x)=\begin{cases}
u(x) & \textup{if $x\in\Omega$}\cr
0
& \textup{if $x\in(a-1,a]$}\cr
L
& \textup{if $x\in(b,b+1]$}
\end{cases}
\qquad
V(x)=\begin{cases}
v(x) & \textup{if $x\in\Omega$}\cr
1
& \textup{if $x\in(a-1,a]$}\cr
1
& \textup{if $x\in(b,b+1]$}
\end{cases}
\]
Note that 
for every  $(u,v)\in L^1(\Omega,\R^2)$ and $\eta\in(0,1)$ we have

\[
\FunctepsDir(u,v)=\Functeps^{\kappa}(U,V,(a-\eta,b+\eta))\,.
\]
Therefore, given $(u_\eps,v_\eps)\to(u,v)$ in $L^1(\Omega,\R^2)$
then $(U_\eps,V_\eps)\to(U,V)$ in $L^1((a-\eta,b+\eta),\R^2)$ and
either by Proposition~\ref{p:lbcompBV}  if $\FailureS\in(0,\infty)$ or
by Proposition~\ref{p:lbcompBV sigma infinito} if $\FailureS=\infty$,
we deduce that
\begin{align*}
\liminf_{\eps\to0}&\FunctepsDir(u_\eps,v_\eps)=
\liminf_{\eps\to0}\Functeps^{\kappa}(U_\eps,V_\eps,(a-\eta,b+\eta))\\
&\geq F_{\,\FailureS}(U,V,(a-\eta,b+\eta))=D_{\,\FailureS}(u,v)\,.
\end{align*}
It is sufficient to show the upper bound inequality for
$u\in GBV(\Omega)$ and $v=1$ $\calL^1$-a.e. on $\Omega$.
In this case, by using the standard density argument we can reduce
to functions which are constant on a neighborhood of the boundary
of $\Omega$ by considering $\lambda\in(0,1)$ and $u_\lambda(x):=U(\frac{a+b}{2}+\lambda(x-\frac{a+b}{2}))$,
for $x\in(a-1,b+1)$, and $v_\lambda=1$  $\calL^1$-a.e. on $\Omega$.
Indeed, we have 
\[
\lim_{\lambda\to 1^-}F_{\,\FailureS}(u_\lambda,v_\lambda)=
D_{\,\FailureS}(u,1)\,.
\]
For this class of functions the upper bound inequality follows from
Propositions~\ref{p:uppertilde sigma finito}  if $\FailureS\in(0,\infty)$ or
from Proposition~\ref{p:diffuseliminf sigmabarra infinito} if $\FailureS=\infty$.

Finally, for every $\eps>0$ the functional $\FunctepsDir$ is coercive. Denote by $(u_\eps,v_\eps)$
a minimizing sequence. By truncation we can assume that $u_\eps\in[0,L]$ $\calL^1$-a.e. on $\Omega$.
Therefore, Theorem~\ref{t:cptness sigma finito} if $\FailureS\in(0,\infty)$
or Theorem~\ref{t:cptness sigma infinito} if $\FailureS=\infty$
provide the $L^1(\Omega,\R^2)$ convergence of  $(u_\eps,v_\eps)$ up to a subsequence not relabeled. The fundamental theorem of $\Gamma$-convergence yields \eqref{e:conv min}.
\end{proof}

\subsection{Corollaries of Theorem~\ref{t:finale}}\label{ss:other regimes}

In this section we collect several consequences of Theorem~\ref{t:finale}
essentially by using only simple comparison arguments. First, we establish Theorem~\ref{t:finale brittle}.
\begin{proof}[Proof of  Theorem~\ref{t:finale brittle}.]
We give the proof in case $\FailureS\in(0,\infty)$, the case $\FailureS=\infty$ being similar
and even simpler.
By assumption $\sfrac{\gamma_\eps}{\eps}\to+\infty$, therefore for every $j\in\N$ and
for $\eps$ sufficiently small depending on $j$ we deduce the pointwise estimate $\Functepstilde\geq\Functeps^{(j)}$, where the latter functionals  are defined as
$\Functeps$ in \eqref{functeps} with $\varphi(\gamma_\eps f^2)$ substituted by
$\varphi(j\eps f^2)$. Thus, we
deduce that $\displaystyle{\Gamma\hbox{-}\liminf_{\eps\to 0}\Functepstilde\geq
\Gamma\hbox{-}\lim_{\eps\to 0}\Functeps^{(j)}}=:F_{\FailureS}^{(j)}$ for every $j\in\N$.
In particular, if $\displaystyle{\Gamma\hbox{-}\liminf_{\eps\to 0}\Functepstilde}(u,v)$
is finite, then $u\in GBV(\Omega)$ and $v=1$ $\calL^1$-a.e. $x\in\Omega$. Furthermore,
Theorem~\ref{t:finale} yields that the energy densities of $F_{\FailureS}^{(j)}$ are given
by the convex envelope of $h_{\FailureS}^{(j)}(t)={\displaystyle{\inf_{\tau\in[0,\infty)}}}
\{\varphi(\frac 1\tau)t^2 +\frac{\varsigma^2}{4}j\tau\}$ for every $t\geq0$ for the bulk term
(cf. \eqref{e:hsigma} and use a reparametrization), $g^{(j)}(s)=g_1\big(j^{\sfrac12}s\big)$
for every $s\geq0$ for the jump term (cf. \eqref{e:lagiyo} and use \eqref{e:dependence g varphi}),
and $(j\varphi'(0^+))^{\sfrac12}\FailureS t$ for every $t\geq0$ for the Cantor term. From the latter, by letting $j\to\infty$ we immediately deduce that $u\in GSBV(\Omega)$.
Moreover, from the equality $g^{(j)}(s)=g_1\big(j^{\sfrac12}s\big)$ (cf. \eqref{e:dependence g varphi}) it is clear that
$g^{(j)}(s)\to 2\Psi(1)\chi_{(0,\infty)}(s)$ as $j\to\infty$ thanks to \eqref{e:Lafigrandeinfi} in Proposition~\ref{p:lepropdig}. Therefore, $u\in SBV(\Omega)$.
Finally, let $t>0$ and  $\tau_j\in[0,\infty)$ be such that $\varphi(\frac1{\tau_j})t^2 +\frac{\varsigma^2}{4}j\tau_j\leq h_{\FailureS}^{(j)}(t)+\sfrac1j$. Then, as $h_{\FailureS}^{(j)}(t)\leq\varphi(\infty)t^2$ for all $j$, $(\tau_j)_j$ is infinitesimal as $j\to\infty$, so that $h_{\FailureS}^{(j)}(t)\to\varphi(\infty)t^2$ as $j\to\infty$. Thus, we have shown that $\displaystyle{\Gamma\hbox{-}\liminf_{\eps\to 0}\Functepstilde}\geq \widetilde{F}$.

The upper bound inequality easily follows from the estimate
$\varphi(\gamma_\eps f^2(t))\leq\chi_{(0,\infty)}(t)$ for every $t>0$, and the construction in Proposition~\ref{p:uppertilde sigma infinito}.
\end{proof}
Note that if $\gamma_\eps=1$ for every $\eps>0$ and $Q$ is strictly increasing,
we recover exactly the Ambrosio and Tortorelli model with any assigned continuous
degradation function $\psi$ choosing $\varphi$ appropriately.

Now we turn to address the case $\FailureS=0$.
\begin{proposition}\label{p:gamma limite sigma 0}
Assume (Hp~$1$)-(Hp~$4$), and \eqref{e:FailureS def} holds with $\FailureS=0$. Then
\begin{equation}\label{e:gamma limite sigmite finito}
\Gamma(L^1)\text{-}\lim_{\eps\to0} \Functeps(u,v)=0
\end{equation}
if $u\in L^1(\Omega)$ and $v=1$ $\calL^1$-a.e. on $\Omega$, and
$\infty$ otherwise on $L^1(\Omega,\R^2)$.
\end{proposition}
\begin{proof}
By a standard density argument and the $L^1(\Omega)$ lower semicontinuity of the $\Gamma$-limsup, to prove the result it is sufficient to establish the upper bound inequality for $u\in BV(\Omega)$ and $v=1$ $\calL^1$-a.e. on $\Omega$.

With this aim, fix $j\in\N$ and define for every $t\in[0,1)$
\[
f^{(j)}(t):=\frac{\hatf(t)}{(j^2\PotDann(1-t))\wedge Q(1-t)}\,.
\]
Notice that, $f^{(j)}(t)\geq f(t)$ for every $t\in[0,1)$ and $j\in\N$; and moreover as $\FailureS=0$ there is $t_j$ such that
$(j^2\PotDann(1-t))\wedge Q(1-t)=j^2\PotDann(1-t)$ for all $t\in[t_j,1)$, 
and $f^{(j)}(t)=f(t)$ for all $t\in[0,t_j]$. In particular, if $\Functeps^{(j)}$
denotes the functional defined as $\Functeps$
with $f^{(j)}$ in place of $f$, we have that
$\Functeps(u,v)\leq\Functeps^{(j)}(u,v)$,
for every $(u,v)\in L^1(\Omega,\R^2)$, and $j\in\N$.
Thus, by Proposition~\ref{p:uppertilde sigma finito} we get that
\[
 \Gamma(L^1)\text{-}\limsup_{\eps\to0} \Functeps(u,1)\leq
\int_\Omega h^{**}_{\sfrac1j}(|u'|)\dx+\frac1j|D^c u|(\Omega)+
\int_\Omega g^{(\sfrac1j)}(|[u]|)\dd\calH^0,
\]
where $g^{(\sfrac1j)}$ is defined as $g$ in \eqref{e:lagiyo} with
$f^{(j)}$ in place of $f$.

Finally, we conclude by dominated convergence by taking
into account that $(h^{**}_{\sfrac1j})_{j\in\N}$ and
$(g^{(\sfrac1j)})_{j\in\N}$ are decreasing in $j$, with
$h^{**}_{\sfrac1j}\to h_0=0$ by \eqref{e:h0approx} in Lemma~\ref{p:proprieta hsigma},
and $g^{(\sfrac1j)}(s)\to 0$ for every $s\in[0,\infty)$ as
$g^{(\sfrac1j)}(s)\leq\sfrac sj\wedge2\Psi(1)$ by item (ii) in Proposition~\ref{p:lepropdig}.
\end{proof}
Finally, we discuss the role of the assumptions $\{\varphi(\infty),\varphi'(0^+)\}\in(0,\infty)$.
Clearly, $\varphi(\infty)>0$ not to have a trivial limit, the other alternatives are dealt with in the next result. Recall that
with $\varphi'(0^+)=\infty$ we mean that the limit of the difference quotient of $\varphi$ in $t=0$ exists and it is not finite, and that
in such a case the corresponding function $h_{\FailureS}^{\ast\ast}$ is superlinear
at infinity (cf. \eqref{e:laaccabroo}). Instead, if $\varphi(\infty)=\infty$, $h_{\FailureS}^{\ast\ast}$ is subquadratic
in the origin (cf. \eqref{e:laaccabrodo})
\begin{corollary}\label{e:varphi variations}
Assume (Hp~$1$) and (Hp~$2$).
\begin{itemize}
\item[(a)] If (Hp~$4$) holds, and $\varphi$ is continuous, non-decreasing with $\varphi(\infty)=\infty$, then
${\displaystyle{\Gamma(L^1)\text{-}\lim_{\eps\to0} \Functeps}}=F_{\FailureS}$
(cf. \eqref{F0}) if $\FailureS\in(0,\infty)$. Instead, if $\FailureS=\infty$
\begin{equation}\label{e:gamma limite varphi = infty sigma finito}
     \Gamma(L^1)\text{-}\lim_{\eps\to0} \Functeps(u,v)=\int_{J_u}g(|[u]|)\dd\calH^0
\end{equation}
if $u\in GSBV(\Omega)$ with $u'=0$ $\calL^1$-a.e. on $\Omega$ and $v=1$ $\calL^1$-a.e. on $\Omega$, $\infty$ otherwise on $L^1$, and $g$ is defined in \eqref{e:lagiyo};

\item[(b)] if (Hp~$3$) holds, $\varphi^{-1}(0)=0$, $\varphi'(0^+)=0$, then for every $\FailureS\in(0,\infty]$
\begin{equation}\label{e:gamma limite varphiprimo = 0}
     \Gamma(L^1)\text{-}\lim_{\eps\to0} \Functeps(u,v)=0
\end{equation}
if $u\in L^1(\Omega)$ and $v=1$ $\calL^1$-a.e. on $\Omega$, and
$\infty$ otherwise on $L^1(\Omega;\R^2)$;

\item[(c)] if (Hp~$3$) holds, $\varphi^{-1}(0)=0$, 
$\varphi'(0^+)=\infty$, then for every $\FailureS\in(0,\infty]$
\begin{equation}\label{e:gamma limite varphiprimo = infty}
     \Gamma(L^1)\text{-}\lim_{\eps\to0} \Functeps(u,v)=
     \int_\Omega h_{\FailureS}^{**}(|u'|)\dx+2\Psi(1)\calH^0(J_u)
\end{equation}
if $u\in SBV(\Omega)$ and $v=1$ $\calL^1$-a.e. on $\Omega$, and
$\infty$ otherwise on $L^1(\Omega,\R^2)$;

\item[(d)] if $\varphi$ is continuous, non-decreasing, with $\varphi^{-1}(0)=0$, and $\varphi(\infty)=\varphi'(0^+)=\infty$, then the equality in \eqref{e:gamma limite varphiprimo = infty}
holds for every $\FailureS\in(0,\infty)$, while for $\FailureS=\infty$
\begin{equation}\label{e:gamma limite varphiprimo e varphi= infty}
     \Gamma(L^1)\text{-}\lim_{\eps\to0} \Functeps(u,v)=2\Psi(1)\calH^0(J_u)
\end{equation}
if $u\in GSBV(\Omega)$ with $u'=0$ $\calL^1$-a.e. on $\Omega$ and $v=1$ $\calL^1$-a.e. on $\Omega$, $\infty$ otherwise on $L^1(\Omega;\R^2)$.
\end{itemize}
\end{corollary}
\begin{proof}
The general strategy is to compare each $\Functeps$ with an auxiliary functional $\Functeps^{(j)}$
either from below or from above according to the case. For every $j\in\N$, we may apply
Theorem~\ref{t:finale} to $(\Functeps^{(j)})_{\eps}$, letting $F_{\,\FailureS}^{(j)}$ be the corresponding
$\Gamma$-limit, we then pass to the limit as $j\to\infty$ to obtain an estimate from above for the $\Gamma$-limsup or from below for the $\Gamma$-liminf with the functional in the corresponding statement.

{\bf Proof of (a).} 
If $\FailureS\in(0,\infty)$, we can argue as in Theorem~\ref{t:finale}.
Indeed, the only difference is that the bulk energy density $h_{\FailureS}^{\ast\ast}$
is sub-quadratic close to the origin (cf. \eqref{e:laaccabrodo}).
Instead, if $\FailureS=\infty$ let $\Functeps^{(j)}$ be obtained substituting $\varphi$ with $j\wedge \varphi(t)$, then ${\displaystyle{\Gamma\hbox{-}\liminf_{\eps\to 0}\Functeps(u,v)\geq F_{\infty}^{(j)}}}$.
Note that the surface energy densities are given by $g$ in \eqref{e:lagiyo} for every $j\in\N$, instead
the bulk energy densities equal to $jt^2$ by item (ii) in Lemma~\ref{p:proprieta hsigma}.
The lower bound then follows. The upper bound is a consequence of Proposition~\ref{p:uppertilde sigma infinito}.

{\bf Proof of (b).} 
Let $\Functeps^{(j)}$ be obtained by substituting $\varphi$ in the definition of $\Functeps$ with the function given by $\sfrac tj$ on the connected component of the set
$\{t\in(0,\infty):\,\varphi(t)< \sfrac tj \}$ whose closure contains the origin, and equal to $\varphi$ otherwise.
Then, ${\displaystyle{\Gamma\hbox{-}\limsup_{\eps\to 0}\Functeps(u,v)\leq F_{\,\FailureS}^{(j)}}}$.
In view of \eqref{e:dependence g varphi}, the surface energy density $g^{(j)}$ of the latter equals $g_1(j^{-\sfrac12}s)$. Thus, $g^{(j)}(s)\to 0$ for every $s\in[0,\infty)$.
Hence, for every $\FailureS\in(0,\infty]$ a rough upper bound for
${\displaystyle{\Gamma\hbox{-}\limsup_{\eps\to 0}\Functeps}}$ is given by
\[
\widetilde{F}_{\FailureS}(u,v)=\varphi(\infty)\int_\Omega|u'|^2\dx
\]
if $u\in GBV(\Omega)$ and $v=1$ $\calL^1$-a.e. on $\Omega$, and $\infty$ otherwise on $L^1$.
The $L^1$ lower semicontinuous envelope of $\widetilde{F}_{\FailureS}$ coincides with the functional on the right hand side of \eqref{e:gamma limite varphiprimo = 0}, as
the class of piecewise constant functions with a finite number of jumps is dense in $L^1$.

{\bf Proof of (c).} Let $\widetilde{\Functu}_{\FailureS}$
be the functional on the right hand side of \eqref{e:gamma limite varphiprimo = infty},
and let $\Functeps^{(j)}$ be obtained substituting $\varphi$ with $j t\wedge \varphi(t)$ for every $j\in\N$.
Then ${\displaystyle{\Gamma\hbox{-}\liminf_{\eps\to 0}\Functeps(u,v)\geq F_{\FailureS}^{(j)}}}$, with
$g^{(j)}(s)=g_1(js)\to 2\Psi(1)\chi_{(0,\infty)}(s)$, and $h_{\FailureS,j}^{**}\leq h_{\FailureS,j+1}^{**}$
with $h_{\FailureS,j}^{**}\to h_{\FailureS}^{**}$ for every $j\in\N$.
Indeed, the latter assertion is trivial if $\FailureS=\infty$ thanks to the identity $h_{\infty,j}^{**}(t)=\varphi(\infty)t^2$ for every $j\in\N$ and $t\geq0$
(cf. item (iii) in Proposition~\ref{p:proprieta hsigma}).
Instead, if $\FailureS\in(0,\infty)$ as $h_{\FailureS,j}(t)=\inf_{[0,\infty)}
\{\textstyle{(\frac1\tau\wedge\varphi(\frac 1{j\tau}))t^2+\frac{\FailureS^2}4j\tau}\}
\leq h_{\FailureS}(t)\leq\varphi(\infty)t^2$, a minimum point $\tau_j$ satisfy $j\tau_j\leq\frac4{\FailureS^2}\varphi(\infty)t^2$. Thus, being $\varphi$ bounded, we have
$h_{\FailureS,j}(t)=\textstyle{\varphi(\frac 1{j\tau_j})t^2+\frac{\FailureS^2}4j\tau_j}\geq h_{\FailureS}(t)$
for $j$ sufficiently large. More precisely, for every $M>0$, $h_{\FailureS,j}=h_{\FailureS}$ on $[0,M]$
for $j$ sufficiently large. The conclusion then follows arguing as to establish
\eqref{e:hinftyapprox} in Proposition~\ref{p:proprieta hsigma}.
Thus, ${\displaystyle{\Gamma\hbox{-}\liminf_{\eps\to 0}\Functeps\geq\widetilde{\Functu}_{\FailureS}}}$,
and the upper bound follows as in Proposition~\ref{p:uppertilde sigma infinito} if $\FailureS=\infty$, and
Proposition~\ref{p:uppertilde sigma finito} otherwise (in the latter case $h_{\FailureS}^{**}$ has superlinear growth at infinity, cf. \eqref{e:laaccabroo} in Lemma~\ref{p:proprieta hsigma}).

{\bf Proof of (d).} If $\FailureS=\infty$ we consider $j\wedge \varphi(t)$ and use the approximation argument in item (a)
to conclude. Instead, if $\FailureS\in(0,\infty)$ we consider $jt\wedge \varphi(t)$ and use the approximation argument in item (c)
to conclude.
\end{proof}


\section*{Acknowledgments} The first author has been supported by the European Union - Next Generation EU, Mission 4 Component 1 CUP G53D23001140006, codice 20229BM9EL, PRIN2022 project: “NutShell - NUmerical modelling and opTimisation of SHELL Structures Against Fracture and Fatigue with Experimental Validations”.
The first author also acknowledges the Italian National Group of Mathematical Physics INdAM-GNFM.

The second and third authors have been supported by the European Union - Next Generation EU, Mission 4 Component 1 CUP B53D23009310006, codice 2022J4FYNJ, PRIN2022 project ``Variational methods for stationary and evolution problems with singularities and interfaces''.
The second and third authors are members of GNAMPA - INdAM.

\bibliographystyle{alpha-noname}
\bibliography{%
  ./Research-Fracture-2023-Cohesive-Focardi-Colasanto-Alessi.bib,%
  ./cfi.bib}

\end{document}